\documentclass[a4paper,11pt]{article}
\usepackage{amsfonts}
\usepackage{graphicx}
\usepackage{epstopdf}
\ifpdf
  \DeclareGraphicsExtensions{.eps,.pdf,.png,.jpg}
\else
  \DeclareGraphicsExtensions{.eps}
\fi

\usepackage{amsthm}

\usepackage{color,url}
\usepackage{enumitem}
\usepackage{graphicx}

\usepackage{graphicx,psfrag}
\usepackage{epic}
\usepackage{amssymb}
\usepackage{afterpage}
\usepackage{epsfig}
\usepackage{pstricks}
\usepackage{stmaryrd,url}
\usepackage{tikz,amsmath}
\usepackage{ulem}
\usepackage{footmisc}
\usepackage{bm}
\usepackage{amsfonts}
\usepackage{mathrsfs}
\usepackage[top=1.0in,bottom=1.0in,left=1.0in,right=1.0in]{geometry}

\usepackage{algorithmic,algorithm,xpatch}
\xpatchcmd{\algorithmic}{\setcounter}{\algorithmicfont\setcounter}{}{}
\providecommand{\algorithmicfont}{}
\providecommand{\setalgorithmicfont}[1]{\renewcommand{\algorithmicfont}{#1}}

\newcommand{\f}{\frac}

\newcommand{\bs}{\boldsymbol}
\newcommand{\wt}{\widetilde}
\newcommand{\vc}{\gamma}
\newcommand{\G}{\mathcal{G}}
\newcommand{\Smc}{\mathcal{S}}
\newcommand{\Proj}{P}
\newcommand{\R}{R}

\newtheorem{definition}{Definition}
\newtheorem{lemma}{Lemma}
\newtheorem{remark}{Remark}
\newtheorem{theorem}{Theorem}

\newcommand{\revi}[1]{{\color{black}#1}}
\newcommand{\revii}[1]{{\color{black}#1}}

\title{A high order finite difference method for the elastic wave equation in bounded domains with nonconforming interfaces}
\begin{document}
\author{Lu Zhang\thanks{Department of Applied Physics and Applied Mathematics, Columbia University, New York, NY 10027, USA. mail: lz2784@columbia.edu}\and Siyang Wang\thanks{Corresponding author. Department of Mathematics and Mathematical Statistics, Ume{\aa} University, 901 87 Ume\aa,  
    Sweden. Email: siyang.wang@umu.se}}
\maketitle

\begin{abstract}
We develop a stable finite difference method for the elastic wave equation in bounded media, where the material properties can be discontinuous at curved interfaces. The governing equation is discretized in second order form by a fourth or sixth order accurate summation-by-parts operator. The mesh size is determined by the velocity structure of the material, resulting in nonconforming grid interfaces with hanging nodes. We use order-preserving interpolation and the ghost point technique to couple adjacent mesh blocks in an energy-conserving manner, which is supported by a fully discrete stability analysis. \revi{In our previous work for the wave equation, two pairs of order-preserving interpolation operators are needed when imposing the interface conditions weakly by a penalty technique. Here, we only use one pair in the ghost point method.} In numerical experiments, we demonstrate that the convergence rate is optimal, and is the same as when a globally uniform mesh is used in a single domain. In addition, with a predictor-corrector time integration method, we obtain time stepping stability with stepsize almost the same as given by the usual Courant–Friedrichs–Lewy condition.
\end{abstract}


\textbf{AMS}\\
  65M06, 65M12

\section{Introduction}

Based on the classical dispersion analysis by Kreiss and Oliger \cite{Kreiss1972}, high order numerical methods solve hyperbolic partial differential equations (PDEs) more efficiently than low order methods for sufficiently smooth problems.  In the finite difference framework, Kreiss and Scherer introduced the concept of summation-by-parts (SBP) \cite{Kreiss1974}, where the difference operators mimic the integration by parts principle through an associated discrete inner product. Since then, SBP operators have extensively been used to derive stable and high order finite difference methods to solve hyperbolic and parabolic PDEs.

In an SBP finite difference method, boundary and interface conditions can be imposed weakly or strongly. The former approach uses the simultaneous approximation term method \cite{Carpenter1994} and is often termed as SBP-SAT. The method uses a penalty technique that is similar to numerical fluxes in a discontinuous Galerkin method \cite{Gassner2013}. For equations with second derivatives in space, boundary and interface conditions can also be imposed strongly by using the ghost point (GP) technique. The SBP-GP method is based on the SBP operator by Sjögreen and Petersson \cite{sjogreen2012fourth} for the second derivative with variable coefficient. This is also the approach used in this paper. 

For a given wave frequency, the wavelength is proportional to the wave speed, which is determined by the material property. In a heterogeneous medium, a smaller wave speed in a subset of the physical domain yields a shorter wavelength localized in the same subset. For computational efficiency, the number of grid points per wavelength shall be kept constant in the entire physical domain \cite{Hagstrom2012,Kreiss1972}. As a consequence, locally refined meshes with nonconforming grid interfaces are present in the discretization. In this case, SBP operators can be used to discretize the equations in each mesh block separated by these interfaces, and solutions in adjacent mesh blocks are coupled together via physical interface conditions. In the coupling procedure, solutions across nonconforming block interfaces are interpolated, and the interpolation must preserve the SBP property of the overall scheme. When the boundary conditions in the spatial direction parallel to the interface are periodic or far-field, standard interpolation based on centered stencils can be used. In bounded domains, however, interpolation stencils must be modified near the edges. A variety of interpolation methods have been proposed, for example, the interpolation operators by Mattsson and Carpenter \cite{Mattsson2010}, the order-preserving (OP) interpolation \cite{Almquist2019}, the projection technique \cite{Kozdon2016}, and the general coupling procedure for advection-diffusion problems \cite{Lundquist2018}.

For the wave equation in second order form, using the interpolation operators by Mattsson and Carpenter reduces the global convergence rate by one order, as compared to the case with conforming grid interfaces \cite{Wang2016}. The accuracy reduction is due to the large interpolation errors near the edges of the nonconforming grid interface.  Motivated by this, OP interpolation operators were constructed in \cite{Almquist2019}. With the model problem of the wave equation on Cartesian meshes with nonconforming grid interfaces, it was observed that the convergence rate was improved to the same order as for the conforming case.  

In this paper, we derive an SBP finite difference method for the elastic wave equations in second order form on curvilinear meshes, with a focus on the numerical coupling at nonconforming grid interfaces. Following \cite{Zhang2021}, we combine SBP operators with ghost points and without ghost points for the spatial discretization. Only SBP operators with the fourth order accurate interior stencil are used in \cite{Zhang2021}. In this work, we will also derive a spatial discretization based on SBP operators with the sixth order accurate interior stencil. For the SBP operators without ghost points, we use the ones constructed by Mattsson \cite{Mattsson2012}. For the SBP operators with ghost points, we use the fourth order version by Sjögreen and Petersson \cite{sjogreen2012fourth}. 
Since the SBP operator with ghost points based on the sixth order interior stencil has not been published, we use the technique from \cite{wang2018fourth} and convert Mattsson's operator to a new  SBP operator with ghost points. At the nonconforming interface, we use interpolation to couple solutions in different mesh blocks. In the SBP-SAT method for the wave equation, two pairs of the OP interpolation operators are used to weakly impose the grid interface conditions. \revii{Here in the SBP-GP method, only one pair of interpolation operators can be used. By a careful analysis of the truncation error, we show which pair shall be used for an optimal convergence rate.} The ghost point values at a nonconforming grid interface are computed by solving a linear system. However, the overall scheme is very efficient because of the favorable time step restriction, \revii{which is supported by a fully discrete stability analysis}.  

In an SBP-SAT discretization of the elastic wave equation, the standard second derivative SBP operator is modified at the boundary to satisfy a compatibility condition with the first derivative SBP operator, usually termed as ``fully compatible''. This procedure simplifies stability analysis but lowers the accuracy of the SBP operator by one order at the boundary and the overall convergence rate by half an order \cite{Almquist2021}. In the SBP-GP discretization, the property ``fully compatible'' is not needed in the stability analysis. In our numerical experiments, we have observed that the convergence rate with nonconforming interfaces on curvilinear meshes is the same as when a globally uniform mesh is used. 

The paper is organized as follows. In Sec.~2, we describe the elastic wave equations on curvilinear meshes. The SBP properties, as well as the OP interpolation operators, are presented in Sec.~3. In Sec.~4, we construct the numerical method and discuss the stability and accuracy properties. In Sec.~5, the predictor-corrector time stepping method is presented. We derive a fully discrete stability analysis and prove energy conservation. Numerical experiments in Sec.~6 verify the accuracy and robustness of the developed method. Finally, we conclude in Sec.~7.

\section{The elastic wave equation}

We consider the time dependent elastic wave equation in a two dimensional bounded domain ${\bf x} \in \Omega$, where ${\bf x} = (x, y)^T$ are the Cartesian coordinates with $x$ denoting the horizontal direction and $y$ denoting the vertical direction.  
The domain $\Omega$ is partitioned into two subdomains $\Omega^f$ and $\Omega^c$, with an interface $\Gamma = \Omega^f {\color{red}\cap} \Omega^c$. Both $\Omega^f$ and $\Omega^c$ have four possibly curved boundaries, i.e. $\Gamma$ may be curved. Without loss of generality, we assume that the wave speed is faster in $\Omega^c$ than in $\Omega^f$. The wavelength is proportional to the wave speed for a given frequency. To keep the number of grid points per wavelength constant, in the discretization the mesh size shall be larger in $\Omega^c$ than in $\Omega^f$, where the superscripts $c$ and $f$ indicating a coarse and fine mesh, respectively.

Our focus is the numerical treatment at the interface on curvilinear grids. To this end, we  introduce the  parameter coordinates ${\bf r} = (r, s)^T$ and smooth one-to-one mappings
\begin{equation}\label{mapping}
{\bf x} = {\bf X}^f({\bf r}) : [0,1]^2 \rightarrow \Omega^f\subset \mathbb{R}^2\ \ \mbox{and}\ \ {\bf x} = {\bf X}^c({\bf r}) : [0,1]^2 \rightarrow \Omega^c\subset \mathbb{R}^2.
\end{equation}
We further assume that the subdomain $\Omega^f$ is located at the top of the subdomain $\Omega^c$. Thus, the interface $\Gamma$ in the parameter coordinate corresponding to $s = 1$ for the coarse subdomain and $s = 0$ for the fine subdomain. 

In $\Omega^c$, the elastic wave equation can be written in curvilinear coordinates as
\begin{align}
\rho^c\frac{\partial^2{\bf C}}{\partial^2 t} = \frac{1}{J^c}\left[\frac{\partial}{\partial r}\left(N_{11}^c\frac{\partial {\bf C}}{\partial r} \right) + \frac{\partial}{\partial r}\left(N_{12}^c\frac{\partial {\bf C}}{\partial s} \right) + \frac{\partial}{\partial s}\left(N_{21}^c\frac{\partial {\bf C}}{\partial r} \right)  + \frac{\partial}{\partial s}\left(N_{22}^c\frac{\partial {\bf C}}{\partial s} \right) \right], \label{elastic_curvi}
\end{align}
where ${\bf C}({\bf r}, t)$ is the displacement vector, $\rho^c$ is the density function. The metric derivatives are  
\begin{equation}\label{metric_definition}
\xi_{11}^c = \frac{\partial r}{\partial x}, \ \ \xi_{12}^c = \frac{\partial s}{\partial x}, \ \ \xi_{21}^c = \frac{\partial r}{\partial y}, \ \ \xi_{22}^c = \frac{\partial s}{\partial y},\ \ (x,y)\in\Omega^c,
\end{equation}
and the Jacobian of coordinate transformation is 
\[J^c = \frac{\partial x}{\partial r}\frac{\partial y}{\partial s} - \frac{\partial x}{\partial s}\frac{\partial y}{\partial r}\in (0,\infty),\ \ (x,y)\in\Omega^c.\] 
The material properties in the curvilinear coordinates are described by
\begin{equation}\label{N_definition}
N_{kj}^c = J^c\sum_{i,l = 1}^2 \xi_{ij}^c M_{il}^c \xi_{lk}^c, \quad k,j=1,2,
\end{equation}
where $M_{il}^c$ describes the material properties on Cartesian grids. \revii{The elastic wave equation on Cartesian grids takes the same form as \eqref{elastic_curvi} with $N_{il}^c$ replaced by $M_{il}^c$, see \cite{duru2014stable} for a detailed coordinate transformation.} In general, $M_{il}^c$ are full matrices satisfying \revi{$M_{il}^c=\big(M_{li}^c\big)^T$}. In addition, $M_{11}$ and $M_{22}$ are symmetric positive definite. In the special but important case of isotropic media, we have
\begin{equation}\label{M_definition}
M_{11}^c = \left(\begin{array}{cc}
2\mu^c+\lambda^c & 0 \\
0 & \mu^c \end{array}\right),\ \ \  M_{12}^c = \left(\begin{array}{cc}
0 & \lambda^c\\
\mu^c & 0 \end{array}\right), \ \ \ \revi{M_{21}^c = (M_{12}^c)^T},\ \ \ M_{22}^c = \left(\begin{array}{ccc}
\mu^c & 0 \\
0 & 2\mu^c+\lambda^c\end{array}\right),
\end{equation}
 where $\lambda^c$ and $\mu^c$ are the first and second Lam{\'{e}} parameters, respectively. 
 
We remark that
\begin{itemize}
\item the matrices $N_{kj}^c$ defined in (\ref{N_definition}) are full matrices even in isotropic media. Hence, wave propagation in isotropic media has anisotropic properties in curvilinear coordinates;
\item the matrices $N_{kk}^c$ are symmetric positive definite and $N_{kj}^c=\big(N_{jk}^c\big)^T$, $k,j=1,2$;
\item the inverse mapping of (\ref{mapping}) is 
computed by the metric relations 
\[J^c \frac{\partial r}{\partial x} = \frac{\partial y}{\partial s}, \ \ J^c \frac{\partial r}{\partial y} = -\frac{\partial x}{\partial s},\ \ J^c \frac{\partial s}{\partial x} = -\frac{\partial y}{\partial r},\  \ J^c \frac{\partial s}{\partial y} = \frac{\partial x}{\partial r},\ \ (x,y)\in\Omega^c;\]
\item the unit outward normal for the interface $\Gamma$ of the subdomain $\Omega^c$ is given by 
\begin{equation}\label{outward_normal}
{\bf n}^{c} = +\left(\frac{\partial s}{\partial x}, \frac{\partial s}{\partial y}\right) \bigg/{\sqrt{\left(\frac{\partial s}{\partial x}\right)^2 + \left(\frac{\partial s}{\partial y}\right)^2}},\ (x,y)\in\Omega^c, 
\end{equation}
where $'+'$ corresponds to $s = 1$.
 \end{itemize}

The elastic wave equation in curvilinear coordinates for the fine subdomain in terms of the displacement vector ${\bf F} = {\bf F}({\bf r}, t)$ can be written as
\begin{align}\label{elastic_curvi_f}
\rho^f \frac{\partial^2{\bf F}}{\partial^2 t} = \frac{1}{J^f}\left[\frac{\partial}{\partial r}\left(N_{11}^f\frac{\partial {\bf F}}{\partial r} \right) + \frac{\partial}{\partial r}\left(N_{12}^f\frac{\partial {\bf F}}{\partial s} \right) + \frac{\partial}{\partial s}\left(N_{21}^f\frac{\partial {\bf F}}{\partial r} \right)  + \frac{\partial}{\partial s}\left(N_{22}^f\frac{\partial {\bf F}}{\partial s} \right) \right],
\end{align}
where the material properties are defined analogously as in the coarse domain. Note that the unit outward normal ${\bf n}^f$ for the interface $\Gamma$ of the subdomain $\Omega^f$ corresponding to $s = 0$ shall be changed to $'-'$ compared with $'+'$ in ${\bf n}^c$.

At the material interface $\Gamma$, we impose suitable interface conditions such that the traction vectors and displacement vectors are continuous,
\begin{equation}\label{interface_cond}
\frac{1}{J^c\Lambda^c}\left({N_{21}^c\frac{\partial {\bf C}}{\partial r} + N_{22}^c\frac{\partial {\bf C}}{\partial s} }\right) = \frac{1}{J^f\Lambda^f}\left({N_{21}^f\frac{\partial {\bf F}}{\partial r} + N_{22}^f\frac{\partial {\bf F}}{\partial s} }\right), \quad {\bf C} = {\bf F},
\end{equation}
where
\begin{equation}\label{lambda_cf}
\Lambda^c = {\sqrt{\left(\frac{\partial s}{\partial x}\right)^2 + \left(\frac{\partial s}{\partial y}\right)^2}}, \ (x,y)\in\Omega^c,\ \ \ \ \Lambda^f = {\sqrt{\left(\frac{\partial s}{\partial x}\right)^2 + \left(\frac{\partial s}{\partial y}\right)^2}}, \ (x,y)\in\Omega^f. 
\end{equation}
Together with suitable physical boundary conditions, such as   Dirichlet boundary conditions or traction free boundary conditions, the problem (\ref{elastic_curvi},\ref{elastic_curvi_f},\ref{interface_cond}) is well-posed \cite{duru2014stable, petersson2015wave}.

\section{Approximation in space}\label{SBP_operators}
In this section, we present the SBP properties for the first and second derivatives, as well as interpolation and restriction operators that will be used for the interface coupling.


\subsection{SBP property}\label{sec_sbp}
Consider a discretization of a bounded domain $x\in [x_l, x_r]$ by $n$ equidistant grid points,
\[
\bs{\wt x} = [x_0, x_1, \cdots, x_n, x_{n+1}],\quad  x_j=x_l+(j-1)h,
\]
 for $j=0,1,\cdots, n, n+1$ and  the grid spacing $h=(x_r-x_l)/(n-1)$. 
 The grid points $x_1$ and $x_n$ coincide with the endpoints $x_l$ and $x_r$, respectively. We also have two ghost points, $x_0$ and $x_{n+1}$, outside the domain. The grid without ghost points is denoted by $\bs{x}=[x_1, \cdots, x_n]$.

We define a discrete inner product 
\begin{equation}\label{iph}
(\bs{u}, \bs{v} )_{hw1} :=h\bs{u}^T W^h \bs{v}=h\sum_{j=1}^n \omega_j u_j v_j,
\end{equation}
where $\bs{u}$ and $\bs{v}$ are grid functions on $\bs{x}$. The positive weights $\omega_j\geq\delta>0$ for some constant $\delta$. In the special case with all weights equal to one, \eqref{iph} reduces to the standard \revi{$l^2$ inner product}.

The SBP property was first introduced in \cite{Kreiss1974} by Kreiss and Scherer, which is restated below. 
\begin{definition}
The difference operator $D$ is a first derivative SBP operator if it satisfies
\begin{equation}\label{sbp_dx}
(\bs{u}, D\bs{v} )_{hw1} = -(D\bs{u}, \bs{v})_{hw1} -u_1v_1 + u_nv_n
\end{equation}
for all grid functions $\bs{u}$ and $\bs{v}$.
\end{definition}

Kreiss and Scherer also constructed a set of first derivative SBP operators. On grid points away from boundaries, the SBP operator $D$ has standard central finite difference stencils with truncation error $\mathcal{O}(h^{2q})$, $q=1,2,3,4$ and the corresponding weights in the scalar product equal to one. To satisfy the SBP property \eqref{sbp_dx}, modified finite difference stencils are used on a few grid points near the boundaries with truncation error  $\mathcal{O}(h^{q})$ and $\omega_j\neq 1$. We say the accuracy order is $(2q,q)$. The SBP operator $D$ does not use any ghost point, and is a square matrix when written in matrix form. 

To solve the elastic wave equation, we also need to approximate the second derivative $\f{d}{dx}(\vc(x) \f{d}{dx})$, where the variable coefficient $\vc(x)$ describes material property and/or coordinate transformation. The simple approach of applying a first derivative SBP operator twice results in an operator with a wide stencil and low accuracy near boundaries. This motivates to construct a second derivative SBP operator with the minimum stencil. 

In \cite{sjogreen2012fourth}, Sjögreen and Petersson derived SBP operators, $\wt G(\vc)$, for the approximation of $\f{d}{dx}(\vc(x) \f{d}{dx})$. We use the tilde symbol to emphasize that $\wt G(\vc)$ uses one ghost point outside each boundary. 
The ghost points in $\wt G(\vc)$ provide extra degrees of freedom to strongly impose boundary and grid interface conditions. We state the SBP property for the second derivative. 

\begin{definition}
The difference operator $\wt G(\vc)$ is a second derivative SBP operator if it satisfies
\begin{equation}\label{sbp_dxx_gp}
(\bs{u},  \wt G(\vc) \bs{v} )_{hw1} = -S_\vc(\bs{u}, \bs{v}) -\vc_1 u_1 \bs{\wt b_1}\bs{v} + \vc_n u_n \bs{\wt b_n}\bs{v}
\end{equation}
for all grid functions $\bs{u}$ and $\bs{v}$. Here, $\vc_1 = \vc(x_1)$, $\vc_n = \vc(x_n)$. The bilinear form $S_\vc(\cdot , \cdot)$ is symmetric positive semidefinite, and does not use any ghost point. The boundary derivative approximations $\bs{\wt b_1}\bs{v}\approx dv(x_1)/dx$ and  $\bs{\wt b_n}\bs{v}\approx dv(x_n)/dx$ use the ghost points.
\end{definition}

\revii{The operator $\wt G(\vc)$ with accuracy order (4,2) was constructed in  \cite{sjogreen2012fourth}.}
On grid points away from the boundary, the standard fourth order accurate central finite difference stencil is used in $\wt G(\vc)$. On the first six grid points, modified stencils are constructed to satisfy the SBP property \eqref{sbp_dxx_gp} with truncation error $\mathcal{O}(h^2)$. The truncation errors of the boundary derivative operators $ \bs{\wt b_1}$ and $\bs{\wt b_n}$ are $\mathcal{O}(h^4)$. In \cite{sjogreen2012fourth}, it was reported that $\wt G(\vc)$ of accuracy order (6,3) and (8,4) exist, but those operators were not published. 

Alternatively,  $\f{d}{dx}(\vc(x) \f{d}{dx})$ can also be approximated by an SBP operator, $\G(\vc)$, that does not use any ghost point. In this case, $\G(\vc)$ satisfies a similar SBP property.
\begin{definition}
The difference operator $\G(\vc)$ is a second derivative SBP operator if it satisfies
\begin{equation}\label{sbp_dxx}
(\bs{u},   \G(\vc) \bs{v} )_{hw1} = -\Smc_\vc(\bs{u}, \bs{v}) -\vc_1 u_1 \bs{b_1}\bs{v} + \vc_n u_n \bs{b_n}\bs{v}
\end{equation}
for all grid functions $\bs{u}$ and $\bs{v}$. Here, $\vc_1 = \vc(x_1)$, $\vc_n = \vc(x_n)$. The bilinear form $\Smc_\vc(\cdot , \cdot)$ is symmetric positive semidefinite. The boundary derivative approximations $\bs{b_1}\bs{v}\approx dv(x_1)/dx$ and  $\bs{b_n}\bs{v}\approx dv(x_n)/dx$.
\end{definition}

In \cite{Mattsson2012}, Mattsson constructed two SBP operators of this type. The first one uses the same fourth order interior stencil as $\wt G(\vc)$ by Sjögreen and Petersson. On the first six grid points, the truncation error is also $\mathcal{O}(h^2)$ but the stencils are different from $\wt G(\vc)$. The boundary derivative approximations are third order accurate. In addition, Mattsson also constructed an operator with truncation error $\mathcal{O}(h^6)$ in the interior and $\mathcal{O}(h^3)$ on the first nine grid points. In this case, the boundary derivative approximations are fourth order accurate. When using the SBP operators without ghost points, boundary and grid interface conditions are often imposed weakly by the SAT method \cite{Carpenter1994}. 

Over the years, the SBP operators $\wt G(\vc)$ and $\G(\vc)$ have been independently used to solve wave propagation problems, see \cite{petersson2015wave,Sjogreen2014} for $\wt G(\vc)$ and \cite{Almquist2021,duru2014stable} for $\G(\vc)$. In \cite{wang2018fourth}, an algorithm was developed to convert $\wt G(\vc)$ by Sjögreen and Petersson to an SBP operator $G(\vc)$ that does not use ghost points. By using the combination of $\wt G(\vc)$ and $G(\vc)$, an improved coupling procedure at nonconforming grid interfaces was developed for the wave equation in \cite{wang2018fourth}, and extended to the elastic wave equations in \cite{Zhang2021}. In both  \cite{wang2018fourth} and \cite{Zhang2021}, periodic boundary conditions are used in the spatial directions tangential to the interface, which allows standard centered interpolation stencils to be used on all hanging nodes. In this paper, we consider nonperiodic boundary conditions and boundary modified interpolation. We derive a fourth order spatial discretization based on the same combination of the (4,2) SBP operators as in \cite{wang2018fourth,Zhang2021}. In addition, we derive a more accurate discretization based on (6,3) SBP operators: for the SBP operator without ghost points, we use the one constructed by Mattsson \cite{Mattsson2012}. We also need a (6,3) SBP operator with ghost points. \revii{Such an operator was constructed in \cite{sjogreen2012fourth} but was not published. Though it is possible to construct one by following the procedure in  \cite{sjogreen2012fourth}, we use a much simpler approach with} the algorithm from \cite{wang2018fourth} to convert the (6,3) $\G(\vc)$ by Mattsson to a (6,3) SBP operator with ghost points.

\newcommand{\Gs}{\mathcal{G}(\vc)}
\newcommand{\bb}{\bs b_1}
\newcommand{\bbt}{{\widetilde{\bs b}}_1}
\newcommand{\df}{{\bs {\widetilde d}}^{(5)}}
\newcommand{\Gst}{\widetilde{\mathcal{G}}(\vc)}
\subsection{A second derivative SBP operator with ghost points of (6,3) accuracy}\label{sec_six}

Let $\Gs$ denote the second derivative SBP operator without ghost points with accuracy order (6,3) \cite{Mattsson2012}, and $\Gst$ denote the corresponding SBP operator with ghost points that is converted from $\Gs$. Since the ghost point is only used for the approximation on the boundary, the only difference between $\Gst$ and $\Gs$ appears in the first and the last row. In the following, we present the procedure of constructing the first row of $\Gst$ corresponding to the left boundary. The stencil at the right boundary is changed in the same way.

Let $V(x)$ be a smooth function in $[x_l,x_r]$ and \revi{define the vector $\bs{v}$ with elements $v_i:=V(x_i)$}.  
The operator associated with $\Gs$ that approximates the first  derivative at the left boundary can be written as 
\[
\bb \bs v = \frac{1}{h}\left(-\frac{25}{12}v_1 + 4v_2 -3v_3 + \frac{4}{3}v_4 - \frac{1}{4}v_5\right)=\frac{dV}{dx}(x_1)+\mathcal{O}(h^4).
\]
The following operator $\df$ is a first order accurate approximation of the fifth derivative at the boundary
\[
\df\bs v = \frac{1}{h^5}\left(-v_0+5v_1-10v_2+10v_3-5v_4+v_5\right)=\frac{d^5V}{dx^5}(x_1)+\mathcal{O}(h).
\]
For any polynomial $V(x)$ of order at most four, the operator $\bb$ is exact with no truncation error.  Also,  $\df\mathbf v=0$ for such polynomials. 
Therefore, $\bbt:=\bb+\beta h^4\df$ is a fourth order approximation of the first derivative at the boundary $x=x_l$ for any real number $\beta$. When $\beta=\frac{1}{4}$, the operator $\bbt$ does not use $v_5$, \revii{and has the minimum width}. We note that $\bbt$ uses the ghost point value $v_0$ and is the associated first derivative boundary operator for $\Gst$. \revii{To see how the boundary closure of $\Gst$ should be modified, we consider a particular grid function $\bs{u}=[1,0,0,\cdots,0]$. Since $\Gst$ and $\Gs$ satisfy the SBP property \eqref{sbp_dxx_gp} and \eqref{sbp_dxx}, respectively, we have
\[
h\omega_1(\Gst \mathbf v)_1 = h\omega_1(\Gs \mathbf v)_1 - \beta h^4 \gamma_1\df\mathbf v,
\]
}where $\omega_1=\frac{13649}{43200}$ is the weight of the SBP norm at the grid point $x=x_1$. Hence, the first row of $\Gst$ is  
\[
(\Gst \mathbf v)_1 = (\Gs \mathbf v)_1 - \frac{\beta h^4}{\omega_1 h}\gamma_1\df\mathbf v.
\]
From the second row, the operator $\Gst$ is the same as $\Gs$. The last row of $\Gst$, as well as the first derivative operator at the right boundary, can be obtained with the same approach.

\subsection{Interpolation and restriction operators}\label{sec_IR}
At a nonconforming interface, interpolation and restriction operators are used to couple the finite difference discretizations on the two sides of the interface. Let the interface be discretized by two uniform grids: a coarse grid with grid spacing $2h$ and a fine grid with grid spacing $h$. Let $\Proj$ and $\R$ be interpolation and restriction operators such that $\Proj \bs{u}_c \approx \bs{u}_f$ and $\R \bs{u}_f \approx \bs{u_c}$. Here, $\bs{u}_c$ and  $\bs{u}_f$ are grid functions with values of a smooth  function $u(x)$ evaluated on the coarse and fine grids, respectively. We note that in some literature, both $P$ and $R$ are referred to as interpolation operators. 


To preserve the SBP property of the overall discretization, the interpolation and restriction must satisfy a \revi{compatibility condition \cite{Mattsson2010}}. 
\begin{definition}
	The interpolation operator $\Proj$ and restriction operator $\R$ are compatible if 
	\begin{equation}\label{PRcomp}
	(\bs{v}_f, \Proj \bs{v}_c)_{hw1} = (\bs{v}_c, \R \bs{v}_f)_{2hw1},
	\end{equation}
	for all grid functions $\bs{v}_c$ on the coarse grid and $\bs{v}_f$ on the fine grid. The inner product $(\cdot,\cdot)_{hw1}$ is defined in \eqref{iph}, and $(\cdot,\cdot)_{2hw1}$ is defined analogously with grid spacing $2h$.
\end{definition}

For problems with periodic boundary conditions in the spatial direction parallel to the interface, all weights in the SBP inner product \revi{are} equal to one. In this case, interpolation and restrictions operators based on centered stencils can easily be constructed to the desired accuracy. For non-periodic problems, the weights in the SBP inner product are modified on a few grid points near the edges of the interface. As a consequence, the interpolation and restriction stencils shall also be modified accordingly so that the compatibility condition \eqref{PRcomp} is satisfied. Such operators were first constructed by Mattsson and Carpenter in \cite{Mattsson2010}. For both interpolation and restriction operators, the truncation error is $\mathcal{O}(h^{2q})$ in the interior and $\mathcal{O}(h^{q})$ on a few grid points near boundaries, where $q=1,2,3,4$. Therefore, the accuracy order of the interpolation and restriction operators from \cite{Mattsson2010}  have the same $(2q,q)$ fashion as the SBP operators. When these operators are used to solve the wave equation in second order form, it was found in \cite{Wang2016} that the overall convergence rate is one order lower than the case with a conforming interface. Based on a truncation error analysis, the convergence rate reduction is due to the $\mathcal{O}(h^q)$ error in interpolation and restriction, even though it is only presented on a few grid points in a two dimensional domain. 

To circumvent the order reduction in convergence rate, order-preserving (OP) interpolation and restriction operators were constructed in \cite{Almquist2019}. More precisely, for each $q=1,2,3,4$, OP interpolation and restriction operators come in two  pairs. \revii{The first pair consists of a projection operator $\Proj$ with accuracy order $(2q,q)$ and a restriction operator $\R$ with accuracy order $(2q,q+1)$, while the second pair consists of a projection operator $\Proj$ with accuracy order $(2q,q+1)$ and a restriction operator $\R$ with accuracy order $(2q,q)$.} Comparing with the operators by Mattsson and Carpenter, the accuracy near the boundary of either $\Proj$ or $\R$ is improved by one order. We note that the order of accuracy of the OP operators is maximized, and it is not possible to derive one pair of $P$ and $R$ with accuracy $(2q,q+1)$ \cite{Lundquist2018}. An SBP-SAT discretization based on these two pairs of OP  operators for the wave equation in  second order form with a nonconforming grid interface was derived in \cite{Almquist2019}, which has the same overall convergence rate for the case with a  conforming interface. 

In the next section, we will also use OP operators in the discretization of the elastic wave equations. We will show that, with interface conditions imposed by the ghost point method, only one pair of OP operators is \revi{needed} to obtain the desired convergence rate: for the discretizations based on operators of (4,2) and (6,3) order of accuracy, the overall convergence rates are 4 and 5, respectively.

\section{The semidiscretization}\label{sec_semidiscretization}
In this section, we use similar notations as our previous work \cite{Zhang2021}, but \revi{adapt} them to 
two dimensional case. Precisely, we present the semidiscretization of the elastic wave equations (\ref{elastic_curvi}) and  (\ref{elastic_curvi_f}), and focus on the numerical treatment at the nonconforming grid interface $\Gamma$.  For the physical boundary conditions, we use the methods in \cite{petersson2015wave}. For a clear presentation, we do not include terms corresponding to boundary conditions in our discretization. \revii{In the numerical experiments, however, we have physical boundary conditions of Dirichlet type and free surface, which are handled in the same way as \cite{petersson2015wave}.} We assume the mesh size in the reference domains satisfy
\[h_1(n_1^h-1) = 1, \ \ \ h_2(n_2^h-1) = 1,\ \ \  2h_1(n_1^{2h}-1) = 1, \ \ \ 2h_2(n_2^{2h}-1) = 1,\]
where $(h_1, h_2)$ and $(2h_1, 2h_2)$ correspond to the fine and coarse reference domains, respectively. That is, the ratio of the mesh sizes is $1:2$. This is an ideal mesh if the wave speed in $\Omega^f$ is half of the wave speed in $\Omega^c$. \revii{This type of mesh is also used in previous literature, for example, Fig. 3 in \cite{wang2018fourth}.} Note that other ratios can be treated analogously. 

We define the following sets of grid points 
\begin{align*}
I_{\Omega^c} &= \{i = 1,2,\cdots,n_1^{2h}, j = 1,2,\cdots,n_2^{2h}\}, \ \ \ I_{\Gamma^c}  = \{i = 1,2,\cdots,n_1^{2h}, j = n_2^{2h}\},\\
I_{\Omega^f} &= \{i = 1,2,\cdots,n_1^h, j = 1,2,\cdots,n_2^h\},\ \ \ \ \ I_{\Gamma^f}  = \{i = 1,2,\cdots,n_1^{h}, j = 1\}.
\end{align*}	
In the above, $I_{\Omega^c}$ and $I_{\Omega^f}$ correspond to the grid points in $\Omega^c$ and $\Omega^f$, respectively. Similarly, $I_{\Gamma^c}$ and $I_{\Gamma^f}$ correspond to the grid points on the interface of  $\Omega^c$ and $\Omega^f$, respectively. In addition, we introduce the multi-index notations
\[{\bf i} = (i,j),\ \ {\bf r}_{\bf i} = (r_i,s_j),\ \ {\bf x}_{\bf i} = (x_i,y_j).\]
Then, from \eqref{mapping} the Cartesian and curvilinear coordinates are related through the mapping ${\bf x}_{\bf i} = {\bf X}^c({\bf r}_{\bf i})$ when ${\bf i}\in I_{\Omega^c}$, and ${\bf x}_{\bf i} = {\bf X}^f({\bf r}_{\bf i})$ when ${\bf i}\in I_{\Omega^f}$. We define ${\bf c}$ and ${\bf f}$ as
\[{\bf c} = [c^{(1)}_{1,1}, c^{(2)}_{1,1}, c^{(1)}_{2,1}, c^{(2)}_{2,1},\cdots]^T, \quad {\bf f} = [f^{(1)}_{1,1}, f^{(2)}_{1,1}, f^{(1)}_{2,1}, f^{(2)}_{2,1},\cdots]^T\] 
for the numerical approximation of ${\bf C}({\bf x})$ and ${\bf F}({\bf x})$, respectively. Elementwise, we have ${\bf c}_{\bf i} \!:=\! \left(c^{(1)}_{\bf i}, c^{(2)}_{\bf i}\right)^T\!\approx\! {\bf C}({\bf x}_{\bf i})$ and ${\bf f}_{\bf i} := \left(f^{(1)}_{\bf i}, f^{(2)}_{\bf i}\right)^T\approx {\bf F}({\bf x}_{\bf i})$. 

In our discretization, we use SBP operators with ghost points in $\Omega^c$ and SBP operators without ghost points in $\Omega^f$ for the numerical coupling at the nonconforming grid interfaces. Hence, the vector $\bf c$ has length $2n_1^{2h}(n_2^{2h}+1)$ and contains the ghost point values. The length of ${\bf f}$ is $2n_1^{h}n_2^h$ because ${\bf f}$ contains no ghost point value. It is also possible to use SBP operators with ghost points in both $\Omega^c$ and $\Omega^f$, but our approach yields a smaller linear system for the ghost point values with a better structure.

In the following, we present the semidiscretization and use the convention that a tilde symbol over an operator indicates that the operator applies to a grid function with ghost points. The elastic wave equation (\ref{elastic_curvi}) in $\Omega^c$  is approximated by
	\begin{equation}\label{elastic_semi_c}
	\left(\mathcal{J}^{2h}_{\rho}\frac{d^2{{\bf c}}}{dt^2}\right)_{\bf i} = \wt{\mathcal{L}}^{2h}_{\bf i} {{\bf c}},\quad {\bf i}\in I_{\Omega^c},\quad t>0,
	\end{equation}
where $\mathcal{J}^{2h}_{\rho}:=({\rho}^{2h}\otimes{\bf I})(J^{2h}\otimes {\bf I})$. Here, 
$\rho^{2h}$ and $J^{2h}$ are $n_1^{2h}n_2^{2h}\times n_1^{2h}n_2^{2h}$ diagonal matrices with the diagonal elements $\rho^{2h}_{\bf i} = \rho^c({\bf x}_{\bf i})$ and $J^{2h}_{\bf i} = J^c({\bf x}_{\bf i})$, ${\bf i}\in I_{\Omega^c}$. The matrix ${\bf I}$ is a $2\times 2$ identity matrix since the spatial dimension of the governing equation is $2$. The discrete spatial operator is
	\begin{equation}\label{L_operator}
	\wt{\mathcal{L}}^{2h} {{\bf c}} = {G}_1^{2h}({N}_{11}^{2h}){\bf c}+\wt{G}_2^{2h}({N}_{22}^{2h}){{\bf c}}+D_1^{2h}({N}_{12}^{2h}{D}_2^{2h}{\bf c}) + D_2^{2h}({N}_{21}^{2h}{D}_1^{2h}{\bf c}),
	\end{equation}
	where the four terms on the right-hand side are given in Appendix \ref{definitions}. We note that the operator $\wt{\mathcal{L}}^{2h}$ uses ghost points through $\wt{G}_2^{2h}({N}_{22}^{2h}){{\bf c}}$.

The elastic wave equation (\ref{elastic_curvi_f}) in $\Omega^f$ is approximated by
		\begin{equation}\label{elastic_semi_f}
		\left(\mathcal{J}^{h}_{\rho}\frac{d^2{{\bf f}}}{dt^2}\right)_{\bf i} = {\mathcal{L}}^{h}_{\bf i} {{\bf f}},\quad {\bf i}\in I_{\Omega^f}\backslash I_{{\Gamma^f}},\quad t>0,
		\end{equation}
		where $\mathcal{J}^{h}_{\rho}$ is defined in the same way as $\mathcal{J}^{2h}_{\rho}$, and the discrete spatial operator is
		\begin{equation}\label{Lf_operator}
		{\mathcal{L}}^{h} {{\bf f}} = {G}_1^{h}({N}_{11}^h){\bf f}+{G}_2^{h}({N}_{22}^h){\bf f}+D_1^{h}({N}_{12}^{h}{D}_2^{h}{\bf f}) + D_2^{h}({N}_{21}^{h}{D}_1^{h}{\bf f}).
		\end{equation}
		Here, $G_1^h(N_{11}^h){\bf f}$ and $D_{l}^h(N_{lm}^hD_m{\bf f})$ are defined similarly as $G_1^{2h}(N_{11}^{2h}){\bf c}$ and $D_{l}^{2h}(N_{lm}^{2h}D_m{\bf c})$ in (\ref{L_operator}), and $G_2^h(N_{22}^h){\bf f}$ is given in Appendix \ref{definitions}. We note that the index ${\bf i}\in I_{\Omega^f}\backslash I_{{\Gamma^f}}$ in \eqref{elastic_semi_f} means that the approximation does not include the grid points on the interface. 

For the  grid points on the interface $\Gamma$ in the domain $\Omega^f$, we compute the numerical solution by using the injection method with a scaled interpolation operator,
		\begin{equation}\label{continuous_sol}
		{\bf f}_{\bf i} = {\mathcal{P}}_{\bf i}({\bf c}),\quad {\bf i}\in I_{\Gamma^f},
		\end{equation}
		which imposes the continuity of the solution at the interface $\Gamma$. The scaled interpolation operator ${\mathcal{P}}$ is related to the OP interpolation operator $P$ via 		
		\begin{equation}\label{scaleP}
		{\mathcal{P}} = (\mathcal{J}_{\Gamma\Lambda}^h)^{-\frac{1}{2}}(\Proj\otimes{\bf I})(\mathcal{J}_{\Gamma\Lambda}^{2h})^{\frac{1}{2}},
		\end{equation}
		where $\mathcal{J}_{\Gamma\Lambda}^{h}:=(J_{\Gamma}^{h}\Lambda^{h})\otimes{\bf I}$ and $\mathcal{J}_{\Gamma\Lambda}^{2h}:=(J_{\Gamma}^{2h}\Lambda^{2h})\otimes{\bf I}$. Here, $J_{\Gamma}^h$ and $\Lambda^h$ are $n_1^{h}\times n_1^{h}$ diagonal matrices with diagonal elements $J_{\Gamma,{\bf i}}^h = J^f({\bf x}_{\bf i})$ and $\Lambda_{\bf i}^h = \Lambda^f({\bf x_i})$, ${\bf i}\in I_{\Gamma^f}$, with $\Lambda^f$ given in (\ref{lambda_cf}); ${J}_{\Gamma}^{2h}$ and ${{\Lambda}^{2h}}$ are $n_1^{2h}\times n_1^{2h}$ diagonal matrices and are defined in the same way as $J_{\Gamma}^h$ and $\Lambda^h$. The OP interpolation operator $\Proj$ has size $n_1^h\times n_1^{2h}$ for scalar grid functions at $\Gamma^c$ as defined in (\ref{PRcomp}). It is clear that the scaled interpolation operator $\mathcal{P}$ has the same order of accuracy as $\Proj$.
		
		In the implementation of our scheme, we use the injection method \eqref{continuous_sol} to obtain the solution on the interface of the fine domain. However, for the stability analysis in  Sec.~\ref{stability}, it is more convenient to use \revi{an equivalent form by taking the time derivative twice of \eqref{continuous_sol}. After that, using \eqref{elastic_semi_c} we obtain }
		\begin{equation}\label{elastic_semi_f_i}
		\left(\mathcal{J}^{h}_{\rho}\frac{d^2{\bf f}}{dt^2} \right)_{\bf i}=
		\mathcal{L}^h_{\bf i}{\bf f} + {\bm \eta}_{\bf i}, \quad {\bf i}\in I_{\Gamma^f}
		\end{equation}
		with 
		\begin{equation}\label{eta}
		{\bm \eta} = \left.\mathcal{J}^{h}_{\rho}\right\vert_{\Gamma} {\mathcal{P}}\left(\left(\mathcal{J}^{2h}_{\rho}\right)^{-1}\left.\wt{\mathcal{L}}^{2h} {\bf c}\right\vert_{\Gamma}\right) - \left.\mathcal{L}^{h}{\bf f}\right\vert_{\Gamma}.
		\end{equation}
		The term $\bm \eta$ is the difference between the spatial approximation at the nonconforming grid interface, and is approximately zero with a truncation error analysis given in Sec.~\ref{accuracy}. Next, we introduce a general notation for the schemes (\ref{elastic_semi_f}) and (\ref{elastic_semi_f_i}) as
		\begin{align}\label{fine_scheme}
		\left(\mathcal{J}^{h}_{\rho}\frac{d^2{\bf f}}{dt^2}\right)_{\bf i} = \hat{\mathcal{L}}^h_{\bf i}{\bf f} = \left\{
		\begin{aligned}
		&\mathcal{L}^h_{\bf i}{\bf f} +{\bm \eta}_{\bf i}, \quad {\bf i}\in I_{\Gamma^f}\\
		&\mathcal{L}^h_{\bf i}{\bf f},\quad\quad\quad {\bf i}\in I_{\Omega^f}\backslash I_{\Gamma^f} 
		\end{aligned}
		\right. ,\quad t > 0.
		\end{align}
		
In addition to the continuity of solution \eqref{continuous_sol}, we also need to numerically impose the continuity of traction at the interface,
		\begin{equation}\label{continuous_traction}
		\left(\left(\mathcal{J}_{\Gamma\Lambda}^{2h}\right)^{-1}\wt{\mathcal{A}}_2^{2h}{\bf c}\right)_{\bf i}
		= {\mathcal{R}}_{\bf i}\Big(\left(\mathcal{J}_{\Gamma\Lambda}^h\right)^{-1}(\mathcal{A}_2^h{\bf f}-h_2\omega_1{\bm \eta})\Big), \quad {\bf i}\in I_{\Gamma^c}.
		\end{equation}
The terms $\wt{\mathcal{A}}_2^{2h}{\bf c}$ and $\mathcal{A}_2^{h}{\bf f}$ are given in Appendix \ref{definitions}. The scaled restriction operator $\mathcal{R}$ is  
			\begin{equation}\label{scaleR}
		{\mathcal{R}} =  (\mathcal{J}_{\Gamma\Lambda}^{2h})^{-\frac{1}{2}}({\R}\otimes{\bf I})(\mathcal{J}_{\Gamma\Lambda}^{h})^{\frac{1}{2}},
		\end{equation}
		where $\R$ is the OP restriction operator of size $n_1^{2h}\times n_1^h$ defined in (\ref{PRcomp}). The term $h_2\omega_1{\bm \eta}$ in \eqref{continuous_traction} is a correction term to the continuity of traction. It is essential for stability, because it cancels out ${\bm \eta}$ in the  spatial discretization \eqref{fine_scheme} in the fine domain. As will be seen later, the compatibility condition (\ref{PRcomp}), as well as the scaling of the interpolation (\ref{scaleP}) and restriction operators (\ref{scaleR}), are also important for stability. 

The equation for the continuity of  solution \eqref{continuous_sol} involves no ghost points. In the \revi{continuity of  traction \eqref{continuous_traction}}, both $\wt{\mathcal{A}}_2^{2h}$ and ${\bm \eta}$ use ghost points. The numerical solutions on the ghost points are coupled together by the restriction operator. 
We rewrite \eqref{continuous_traction}  as a linear system 
\begin{equation}\label{Mass}
M{\mathbf x}={\mathbf b},
\end{equation}
where ${\mathbf x}$ contains the ghost points values. \revii{For convenience, we scale $M$ and $\mathbf{b}$ such that the elements in $M$ are positive on the diagonal and have no meshsize dependent factor.} 

\revii{In general, the explicit form of $M$ contains lengthy expressions that depend on the material parameters as well as interpolation and restriction stencils, and do not provide much insight into its structure. In the case of constant coefficients and  Cartesian grids, we can separate the interpolation and restriction stencils from the material parameters in $M$. As an example, consider  $\Omega^c = [0, l_1]\times[0, \alpha l_2]$ and $\Omega^f = [0, l_1]\times[\alpha l_2, l_2]$ for some $\alpha \in (0,1)$, we have
	\begin{equation*}
	N_{22}^c = \frac{l_1}{\alpha l_2}\begin{pmatrix}
	\mu^c & 0\\
	~0& 2\mu^c+\lambda^c
	\end{pmatrix}, \quad \Lambda^c = \frac{1}{\alpha l_2}, \quad \Lambda^f = \frac{1}{(1-\alpha) l_2}, \quad J^c = \alpha l_1 l_2,\quad  J^f = (1-\alpha)l_1l_2.
	\end{equation*}
From \eqref{continuous_traction}, after some manipulation, we have
	\begin{align}
M &= \frac{\beta}{2}\frac{1}{\alpha l_2}\frac{1-\alpha}{\alpha}\frac{\rho^f}{\rho^c} M^1\! \otimes\! \mbox{diag}([\mu^c, 2\mu^c + \lambda^c]) + \frac{\beta}{\alpha l_2} \mbox{diag}([1,1,\cdots,1]_{1\times n_1^{2h}}) \otimes \mbox{diag}([\mu^c, 2\mu^c + \lambda^c])\label{M1a}\\
&= \frac{\beta}{2}\frac{1}{\alpha l_2}\frac{1-\alpha}{\alpha}\frac{\rho^f}{\rho^c}\left(M^1+ \frac{\alpha}{1-\alpha}\frac{2\rho^c}{\rho^f} \mbox{diag}([1,1,\cdots,1]_{1\times n_1^{2h}})\right) \otimes \mbox{diag}([\mu^c, 2\mu^c + \lambda^c]).\notag
\end{align}
In (\ref{M1a}), the first term  corresponds to the right hand side of (\ref{continuous_traction}), and the second term corresponds to the left hand side of (\ref{continuous_traction}). The matrix $M^1$  depends only on the interpolation and restriction stencils, and $\beta = \frac{1}{4}$ is derived in Section \ref{sec_six}. If $M^1$ is diagonally dominant, then $M$ is also  diagonally dominant. 

In Figure \ref{diagonally_dom}, we show the quantity $M^1_{ii}-\sum_{j\neq i}|M^1_{ij}|$ when $n_1^{2h}=41$ and 81 for $q = 2$ and $q = 3$. It is obvious that when $q=2$, $M^1$ is strictly row-diagonally dominant, and consequently, $M$ is strictly row-diagonally dominant. Similarly, we have also verified that $M^1$ and $M$ are strictly column-diagonally dominant. This result holds for other values of $n_1^{2h}$, because the nonzero elements in the interior rows of $M$ are repeated. In addition, the result can also be generalized to problems with variable coefficients and problems on curvilinear grids, where the ghost points corresponding to the two components of the displacement field are coupled together. In that case, $M$ becomes block-diagonally dominant.

When $q = 3$, $M^1$ is not row-diagonally dominant, but the eigenvalues of $M^1$ shown in Table \ref{tab_eigen} are positive, with the smallest and largest eigenvalues independent of the mesh size. In both cases, $M$ is invertible.
}

\begin{figure}[htbp]
	\centering
	\includegraphics[width=0.48\textwidth,trim={0.5cm 0.cm 1cm 0.1cm}, clip]{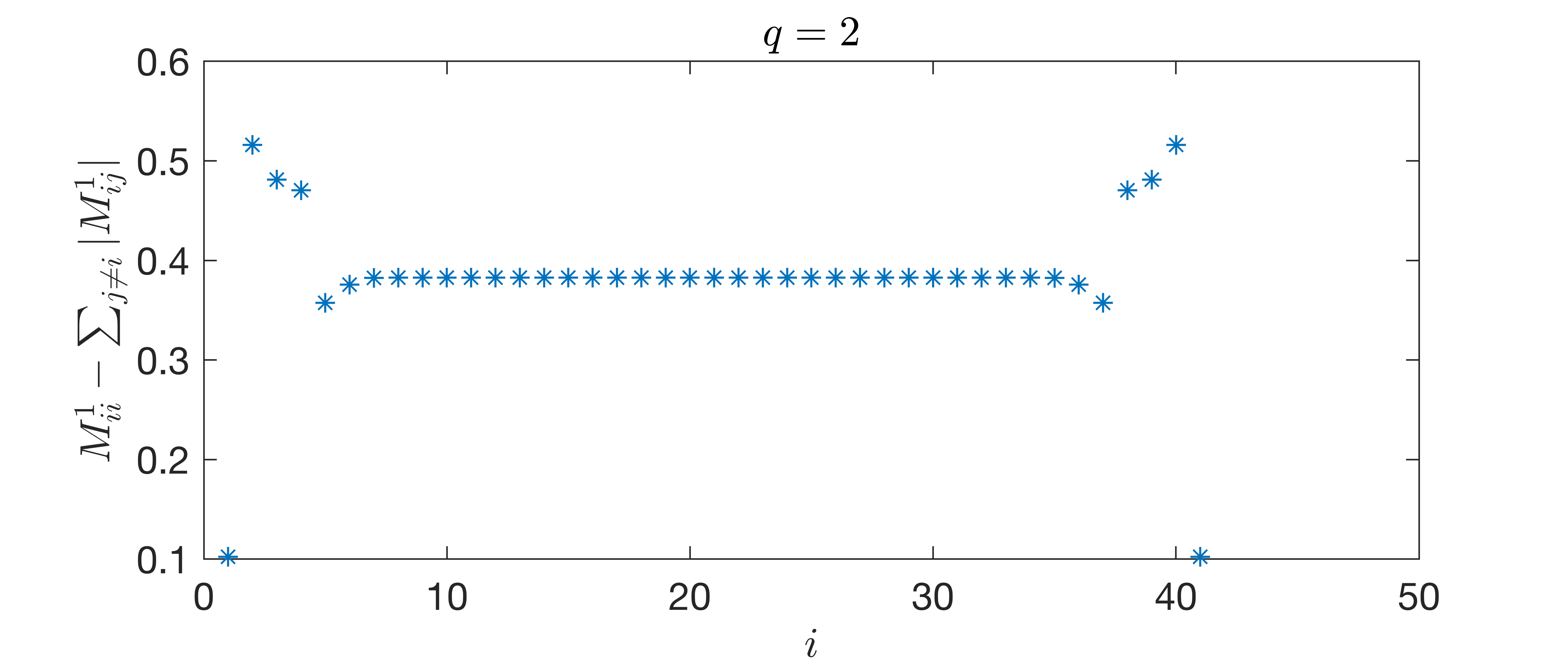}
		\includegraphics[width=0.48\textwidth,trim={0.5cm 0.cm 1cm 0.1cm}, clip]{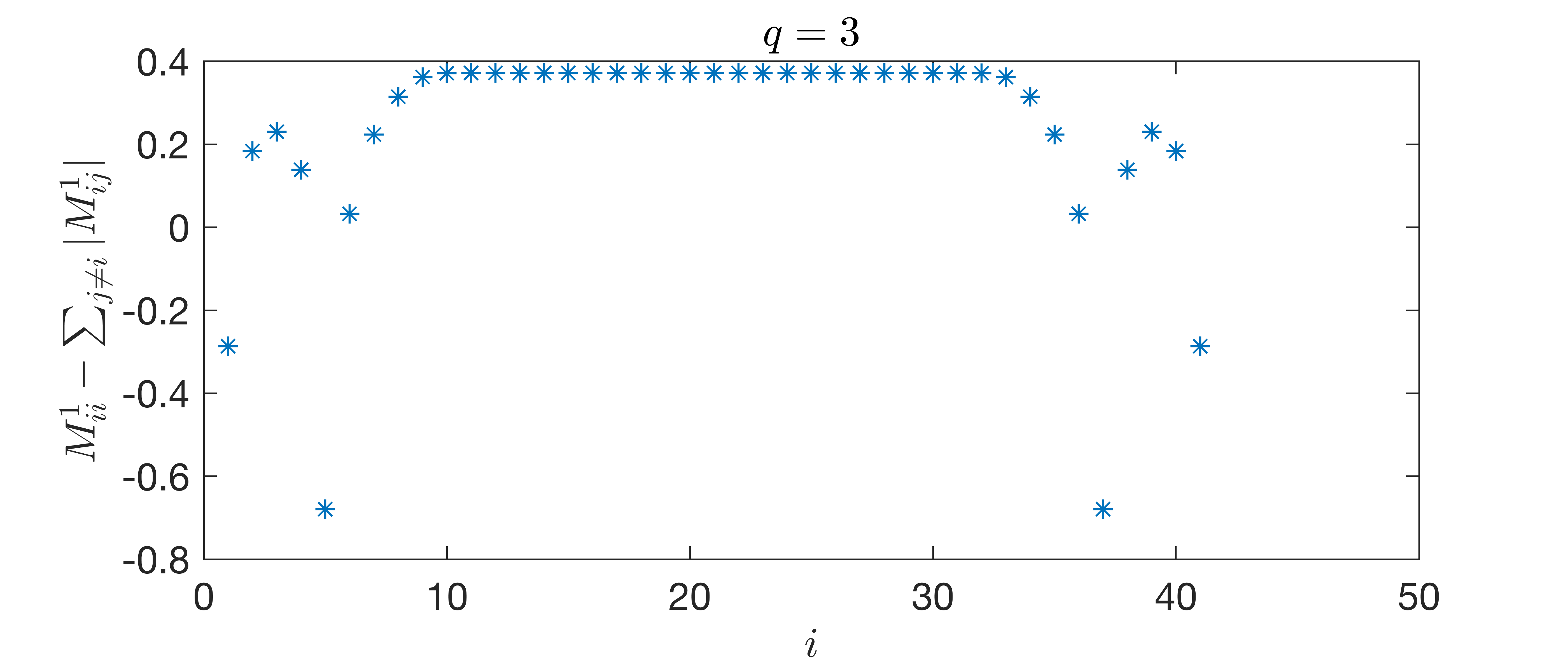}\\
		\includegraphics[width=0.48\textwidth,trim={0.5cm 0.cm 1cm 0.1cm}, clip]{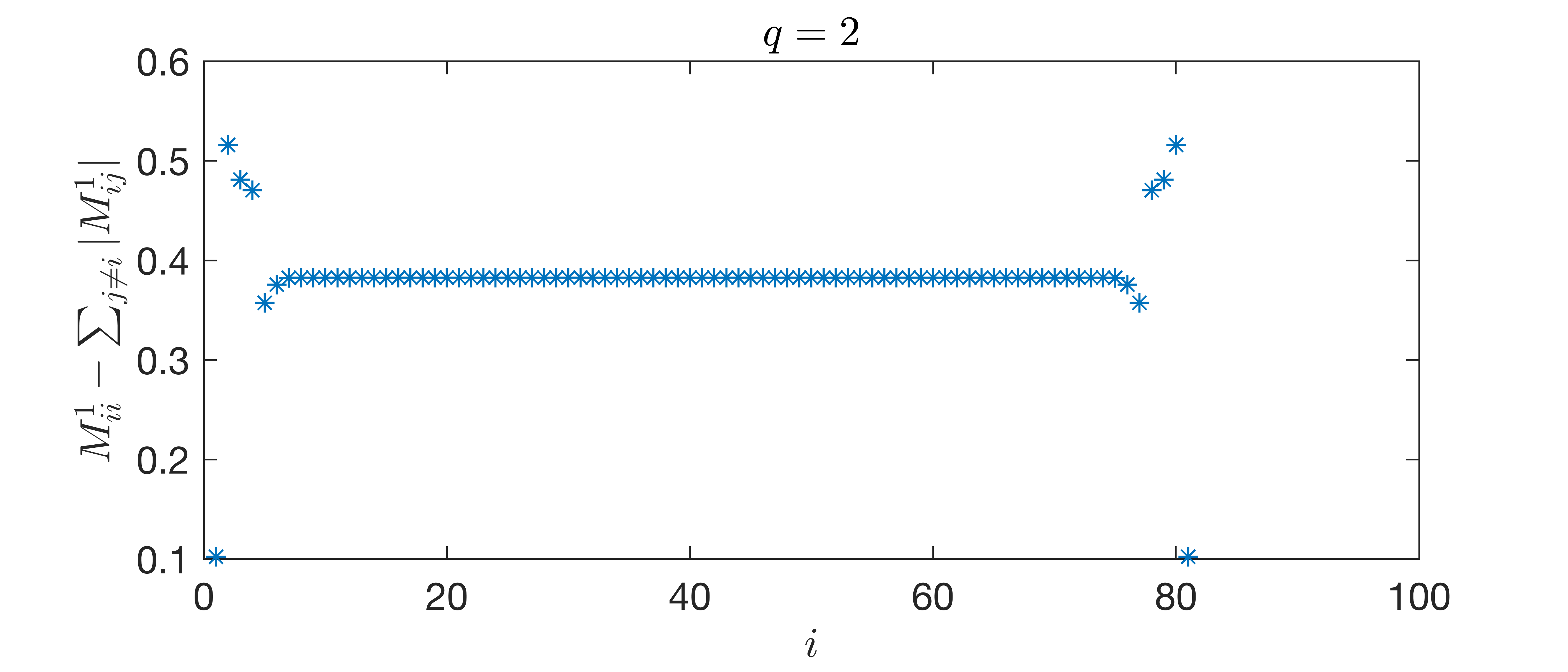}
		\includegraphics[width=0.48\textwidth,trim={0.5cm 0.cm 1cm 0.1cm}, clip]{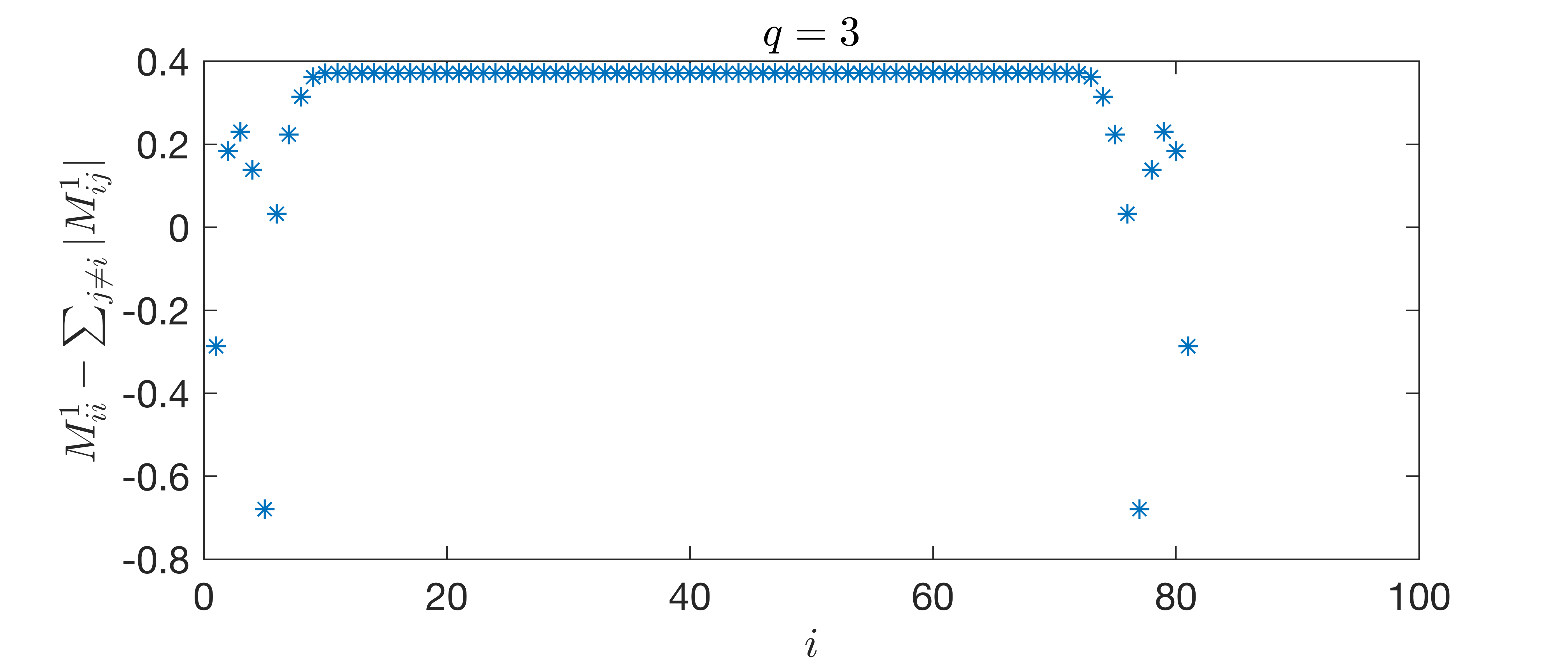}
	\caption{The quantity $M^1_{ii}-\sum_{j\neq i}|M^1_{ij}|$ when $n_1^{2h} = 41$ (top panel) and $n_1^{2h} = 81$ (bottom panel). 
	}\label{diagonally_dom}
\end{figure}

\begin{table}
	\centering
	\begin{tabular}{llllll}
		\hline
		$n_1^{2h}$ & $\lambda_{\mbox{min}} (M^1)$& $\lambda_{\mbox{max}} (M^1)$ & $n_1^{2h}$  & $\lambda_{\mbox{min}} (M^1)$& $\lambda_{\mbox{max}} (M^1)$\\
		\hline
		41 & $0.202175911164221$ & $2.339317314764304$&  321 & $0.202175913794549$& $2.339317314764306$\\
		81 & $0.202175913794549$ & $2.339317314764305$& 641&$0.202175913794548$  & $2.339317314764306$\\
		161 & $0.202175913794549$ &$2.339317314764303$ & 1281& $0.202175913794546$& $2.339317314764316$\\
		\hline
	\end{tabular}
	\caption{The minimal and maximal eigenvalues of $M^1$ in (\ref{M1a}) when $q = 3$.}
	\label{tab_eigen}
\end{table}

\subsection{Accuracy}\label{accuracy}


When using SBP operators with the order of accuracy $(2q,q)$, the convergence rate of the discretization is often higher than $q$, a phenomenon termed as \textit{gain in convergence} \cite{Gustafsson1975,Svard2019}.  For equations with a second derivative in space, an optimal gain is two orders. The precise gain depends on the equation and numerical boundary and interface treatments. Indeed, the gain in convergence could be higher or lower than two orders \cite{Wang2017}. With nonconforming grid interfaces, the interpolation and restriction operators also affect the convergence rate. While the convergence analysis for the full method is out of the scope of this work, we shall analyze the truncation error to identify the precise accuracy requirement for the interpolation and restriction operators.  For 2D problems discretized on a grid with $\mathcal{O}(n^2)$ grid points, to obtain the optimal convergence rate $\min(2q,q+2)$, we may have at most $\mathcal{O}(n)$ grid points with truncation error $\mathcal{O}(h^q)$, and $\mathcal{O}(1)$ grid points with truncation error $\mathcal{O}(h^{q-1})$ \cite{Wang2018b}. Note that we have dropped the subscripts and superscripts of $n$ and $h$ in the big $\mathcal{O}$ notation.


The truncation errors of the semidiscretizations \eqref{elastic_semi_c} and \eqref{fine_scheme} are determined by the operators $\wt{\mathcal{L}}^{2h}$ and $ \hat{\mathcal{L}}^h$.   Sufficiently far away from boundaries and the interface $\Gamma$, standard central finite difference approximations with truncation error  $\mathcal{O}(h^{2q})$ are used. On the grid points close to the interface but not on the interface,  SBP boundary closures with truncation error $\mathcal{O}(h^{q})$ are used to approximate the spatial derivatives, and the number of such grid points is $\mathcal{O}(n)$. In both cases, the truncation error is not affected by the numerical interface treatment, and has the desired order for the optimal convergence rate. For the grid points on the interface, the truncation error is determined by the SBP operators, the interpolation/restriction operators as well as the ghost point values. Since the truncation errors of the SBP operators and the interpolation/restriction operators are known by construction, we only need to analyze the errors on the ghost points. We consider  \eqref{continuous_traction} in the following form so that both the left-hand side and right-hand side approximate to a vector of zeros,

\begin{equation}\label{continuous_traction_move}
		\left(\left(\mathcal{J}_{\Gamma\Lambda}^{2h}\right)^{-1}\wt{\mathcal{A}}_2^{2h}{\bf c}\right)_{\bf i}
		- {\mathcal{R}}_{\bf i}\Big(\left(\mathcal{J}_{\Gamma\Lambda}^h\right)^{-1}\mathcal{A}_2^h{\bf f} \Big)   =   -{\mathcal{R}}_{\bf i}\Big(h_2\omega_1(\mathcal{J}_{\Gamma\Lambda}^h)^{-1}{\bm \eta}\Big), \quad {\bf i}\in I_{\Gamma^c}.
		\end{equation} 
To analyze the errors on the ghost points, we consider ${\bf c}$ and ${\bf f}$ as sufficiently smooth functions evaluated on the grid, except on the ghost points. We then determine the errors on the ghost points. 
		
Left-hand side of \eqref{continuous_traction_move}: 1) The truncation error of $\wt{\mathcal{A}}_2^{2h}{\bf c}$ is determined by the first derivative approximations as well as the ghost point values, whose truncation errors are denoted by $\mathcal{O}(h^{d})$ and $\mathcal{O}(h^{g})$, respectively. 
Hence, the  truncation error of $\wt{\mathcal{A}}_2^{2h}{\bf c}$ is  $\mathcal{O}(h^{d})+\mathcal{O}(h^{g-1})$. The order $g-1$ is due to the $1/h$ factor in $\wt{\mathcal{A}}_2^{2h}$, \revii{which involves first derivative approximation}; 2) In the second term, the values of ${\bf f}$ on the interface is computed by \eqref{continuous_sol} using the interpolation operator. As a consequence, \revi{the truncation error of $\mathcal{A}_2^h{\bf f}$} is $\mathcal{O}(h^{d})+\mathcal{O}(h^{p-1})$. Taking into account of the restriction operator, the truncation error of all terms on the left-hand side is $\mathcal{O}(h^{d})+\mathcal{O}(h^{g-1})+\mathcal{O}(h^{p-1})+\mathcal{O}(h^{r})$. We note that the interpolation/restriction error $\mathcal{O}(h^{p, r})$ in the interior of $\Gamma$ is different from that near the edges, i.e. $p$ and $r$ depend on index ${\bf i}$.

Right-hand side of \eqref{continuous_traction_move}:  The truncation error of $\wt{\mathcal{L}}^{2h}{\bf c}$ and ${\mathcal{L}}^{h}{\bf f}$ are $\mathcal{O}(h^{q})+\mathcal{O}(h^{g-2})$ and $\mathcal{O}(h^{q})+\mathcal{O}(h^{p-2})$. Then, the truncation error of ${\bm \eta}$ is  $\mathcal{O}(h^{q})+\mathcal{O}(h^{g-2})+\mathcal{O}(h^{p-2})$. \revi{Due to $h_2$, the term ${h_2\omega_1(\mathcal{J}_{\Gamma\Lambda}^h)^{-1}\bm \eta}$ is $\mathcal{O}(h^{q+1})+\mathcal{O}(h^{g-1})+\mathcal{O}(h^{p-1})$. Since the number of nonzero elements in each row of the restriction operator is independent of $h$, when the restriction operator is applied to ${h_2\omega_1(\mathcal{J}_{\Gamma\Lambda}^h)^{-1}\bm \eta}$, it} gives a truncation error of all terms on the right-hand side  $\mathcal{O}(h^{q+1})+\mathcal{O}(h^{g-1})+\mathcal{O}(h^{p-1})$. It is important to note that when the restriction operator acts on ${\bm \eta}\sim 0$, it does not introduce additional truncation error.

\revii{Next, we move all terms $\mathcal{O}(h^{g-1})$ to the left-hand side, and the other terms to the right-hand side. We write the error equation in matrix form 
\begin{equation}\label{eqn_e}
M{\bf e}={\bf T}, 
\end{equation}
where we have multiplied by $h$ on both sides, so that the error on the ghost point, ${\bf e}$, is $\mathcal{O}(h^{g})$. Consequently, we have 
 ${\bf T}=\mathcal{O}(h^{d+1})+\mathcal{O}(h^{p})+\mathcal{O}(h^{r+1})+\mathcal{O}(h^{q+2})$}. The matrix $M$ is exactly the same as in \eqref{Mass}. We know $d\geq q+1$ and $p+r\leq 2q+1$. To make the right-hand side as small as possible, we choose $p=q+1$ and $r=q$, i.e. the interpolation operator is one order more accurate than the restriction operator near the edges of the interface. As a consequence, ${\bf T}$ is $\mathcal{O}(h^{q+2})$ for the ghost points in the interior, and $\mathcal{O}(h^{q+1})$ on a few ghost points near the edges of the interface. 
To get the order $g$, we need to solve \eqref{eqn_e} and analyze the component of ${\bf e}=M^{-1}{\bf T}$. For strictly diagonally dominant $M$, we have the following results.

\begin{lemma}\label{Mr}
If $M$ is strictly row-diagonally dominant, then $M^{-1}$ exists and 
\[
|M^{-1}|_\infty\leq \frac{1}{\min_i (M_{ii}-\sum_{j\neq i} |M_{ij}|)}:=C_\infty.
\]
\end{lemma}
\begin{proof}
See Theorem 1 in \cite{Ahlberg1963}.
\end{proof}
An immediate consequence of Lemma \ref{Mr} is 
\begin{lemma}\label{Mc}
If $M$ is strictly column-diagonally dominant, then $M^{-1}$ exists and  
\[
|M^{-1}|_1\leq \frac{1}{\min_j (M_{jj}-\sum_{i\neq j} |M_{ij}|)}:= C_1.
\]
\end{lemma}
\begin{proof}
It follows from $|M^{-1}|_1=|(M^{-1})^T|_\infty=|(M^{T})^{-1}|_\infty$ and Lemma \ref{Mr}.
\end{proof}

With the above two lemmas, we have the following result for the error on the ghost points. 
\begin{lemma}\label{lemma_e}
Assume that $M$ in \eqref{eqn_e} is row and column-diagonally dominant, then the solution to \eqref{eqn_e} has $\mathcal{O}(n)$ components of $\mathcal{O}(h^{q+2})$, and $\mathcal{O}(1)$ components of $\mathcal{O}(h^{q+1})$.
\end{lemma}

\begin{proof}

By Lemma \eqref{Mr}, we have \revii{$|{\bf e}|_\infty=|M^{-1}{\bf T}|_\infty\leq |M^{-1}|_\infty |{\bf T}|_\infty\leq C_\infty h^{q+1}$}. Similarly, by Lemma \eqref{Mc}, we have  \revii{$|{\bf e}|_1=|M^{-1}{\bf T}|_1\leq |M^{-1}|_1 |{\bf T}|_1\leq C_1 h^{q+1}$}. Therefore, the number of   $\mathcal{O}(h^{q+1})$ components in ${\bf e}$ cannot depend on $n$. 
\end{proof}

By Lemma \ref{lemma_e}, the truncation error of the semidiscretization \eqref{elastic_semi_c} and \eqref{fine_scheme} on $\Gamma$ is  $\mathcal{O}(h^{q-1})$ on $\mathcal{O}(1)$ points, and $\mathcal{O}(h^{q})$ on $\mathcal{O}(n)$ points. This is exactly the truncation error we would like to have for an optimal convergence rate $\min(2q,q+2)$. In the numerical experiments, we confirm that the optimal convergence rate is obtained. 
The conclusion from the above truncation error analysis is that in the semidiscretization, we choose the OP interpolation operator with accuracy $(2q, q+1)$ and the restriction operator with accuracy $(2q,q)$. 


\section{Discretization in time}\label{sec_time}
We discretize in time on a uniform grid $t_n = n\Delta_t, n = 0,1,\cdots$ with stepsize $\Delta_t > 0$ by an explicit fourth order accurate predictor-corrector time integration method \cite{sjogreen2012fourth}. 

\subsection{Predictor-corrector time-stepping algorithm}\label{time_integrator}
The description of the predictor-corrector method can be found in \cite{Zhang2021}. For completeness, we also present the full discretization below. 
Let ${\bf c}^n$ and ${\bf f}^n$ denote the numerical approximations of ${\bf C}({\bf x}, t_n), {\bf x}\in \Omega^c$ and ${\bf F}({\bf x}, t_n), {\bf x}\in \Omega^f$, respectively. \revii{Note that $\bf c$ contains the ghost point values while $\bf f$ does not.} Assuming that ${\bf c}^n, {\bf c}^{n-1}, {\bf f}^n$ and ${\bf f}^{n-1}$ satisfy the discretized interface conditions (\ref{continuous_sol}) and (\ref{continuous_traction}), the procedure of computing ${{\bf c}^{n+1}}$ and ${\bf f}^{n+1}$ is \revii{given in Algorithm \ref{algorithm1} }.
\setalgorithmicfont{\footnotesize}
\begin{algorithm}\label{algorithm1}
	\caption{Predictor-corrector time-stepping algorithm}
	\begin{algorithmic}[1]
		\STATE Compute the numerical solution at the predictor step on the interior grid points
		\begin{align}
		{\bf c}_{\bf i}^{\ast,n+1} &= 2{\bf c}_{\bf i}^{n} - {\bf c}_{\bf i}^{n-1} + \Delta_t^2\big(\mathcal{J}_{\rho}^{2h}\big)^{-1}\wt{\mathcal{L}}_{\bf i}^{2h}{\bf c}^n,\ \ {\bf i}\in I_{\Omega^c},\label{predictor_c}\\
		{\bf f}_{\bf i}^{\ast,n+1} &= 2{\bf f}_{\bf i}^{n} - {\bf f}_{\bf i}^{n-1} + \Delta_t^2\big(\mathcal{J}_{\rho}^h\big)^{-1}\hat{\mathcal{L}}_{\bf i}^{h}{\bf f}^n,\ \ {\bf i}\in I_{\Omega^f}/I_{\Gamma^f}.\label{predictor_f}
		\end{align}
		Note that \eqref{predictor_f} does not include the grid points on the interface. 
		\STATE Compute the predictor ${\bf f}_{\bf i}^{\ast, n+1}$ on the interface by the injection method
		\[{\bf f}_{\bf i}^{\ast, n+1} = \mathcal{P}_{\bf i}({\bf c}^{\ast, n+1}),\ \ {\bf i}\in I_{\Gamma^f},\] 
		which imposes the continuity of solution at time $t=t_{n+1}$.
		\STATE Compute the predictor  ${\wt{\bf c}}^{\ast, n+1}$ on the ghost points  by solving the linear system 
			\begin{equation}\label{ls1}
		\Big(\big(\mathcal{J}_{\Gamma\Lambda}^{2h}\big)^{-1}\wt{\mathcal{A}}_2^{2h}{\bf c}^{\ast,n+1}\Big)_{\bf i} = \mathcal{R}_{\bf i}\Big(\big(\mathcal{J}_{\Gamma\Lambda}^h\big)^{-1}\big({\mathcal{A}}_2^{h}{\bf f}^{\ast,n+1} - h_2\omega_1{\bm \eta}^{\ast, n+1}\big)\Big),\ \ {\bf i}\in I_{\Gamma^c},
		\end{equation}
		which imposes the continuity of traction. This completes the predictor step. 
		\STATE Evaluate the acceleration at all grid points
		\[{\bf a}_c^n = \frac{{\bf c}^{\ast, n+1} - 2{\bf c}^n + {\bf c}^{n-1}}{\Delta_t^2},\ \ {\bf a}_f^n = \frac{{\bf f}^{\ast, n+1} - 2{\bf f}^n + {\bf f}^{n-1}}{\Delta_t^2}.\]
		\STATE Compute the corrector on  the interior grid points
		\begin{align}
		{\bf c}_{\bf i}^{n+1} &= {\bf c}_{\bf i}^{\ast, n+1} + \frac{\Delta_t^4}{12}\big(\mathcal{J}_\rho^{2h}\big)^{-1}\wt{\mathcal{L}}_{\bf i}^{2h}{\bf a}_c^n,\ \ {\bf i}\in I_{\Omega^c},\label{corrector_c}\\
		{\bf f}_{\bf i}^{n+1} &= {\bf f}_{\bf i}^{\ast, n+1} + \frac{\Delta_t^4}{12}\big(\mathcal{J}_\rho^h\big)^{-1}\hat{\mathcal{L}}_{\bf i}^{h}{\bf a}_f^n,\ \ {\bf i}\in I_{\Omega^f}/I_{\Gamma^f}.\label{corrector_f}
		\end{align}
		\STATE The corrector ${\bf f}_{\bf i}^{n+1}$ on the interface are computed by using the continuity of the solution
		\begin{equation*}
		{\bf f}_{\bf i}^{n+1} = \mathcal{P}_{\bf i}({\bf c}^{n+1}),\ \ {\bf i} \in I_{\Gamma^f}.
		\end{equation*}
		\STATE The corrector ${{\bf c}}^{n+1}$ on the ghost points are computed by using the continuity of traction
		\begin{equation}\label{ls2}
		\Big(\big(\mathcal{J}_{\Gamma\Lambda}^{2h}\big)^{-1}\wt{\mathcal{A}}_2^{2h}{\bf c}^{n+1}\Big)_{\bf i} = \mathcal{R}_{\bf i}\Big(\big(\mathcal{J}_{\Gamma\Lambda}^h\big)^{-1}\big({\mathcal{A}}_2^{h}{\bf f}^{n+1} - h_2\omega_1{\bm \eta}^{n+1}\big)\Big),\ \ {\bf i}\in I_{\Gamma^c}.		\end{equation}
		This completes the corrector step. 
	\end{algorithmic}
\end{algorithm}

At each time step, two systems of linear equations \eqref{ls1} and \eqref{ls2} must be solved. In matrix form, they have the same coefficient matrix \revi{of size $n_1^{2h}\times n_1^{2h}$. The linear systems can be solved efficiently by using LU factorization. For problems in three space dimension, the size of the linear system is $n_1^{2h} n_2^{2h}\times n_1^{2h} n_2^{2h}$, and is often solved by iterative methods. }

\subsection{Fully discrete stability analysis}\label{stability}
In this section, we prove fully discrete stability for the discretization in Sec.~\ref{time_integrator}. More precisely, we show that a fully discrete energy is conserved. 
\subsubsection{Preliminary} 
We start by presenting important properties of the discrete operators that will be used in the stability analysis. First, we need to define discrete inner products for the grid functions on the interface. Let ${\bf u}_{{\Gamma}h}$ and ${\bf v}_{{\Gamma}h}$ be vectors of length $2n_1^{h}$ (the factor $2$ here is due to the spatial dimension), we define the weighted discrete inner product 
\begin{equation}\label{scalar_product_discrete_interface_f}
\left<{\bf u}_{{\Gamma}h}, {\bf v}_{{\Gamma}h}\right>_{hw} = h_1 {\bf u}^T_{{\Gamma}h} \mathcal{W}_\Gamma^{h}\mathcal{J}_{\Gamma\Lambda}^{h} {\bf v}_{{\Gamma}h},
\end{equation}
where  $\mathcal{W}_\Gamma^{h} := W_\Gamma^{h}\otimes{\bf I}$, and the diagonal matrix $W_\Gamma^{h}$ of size $n_1^h\times n_1^h$ contains weights that are the same as the weights in the inner product associated with the SBP operators introduced in \eqref{iph}. The matrix ${\bf I}$ is again the 2-by-2 diagonal matrix. Similarly, we also define a discrete inner product on the interface in the coarse domain 
\begin{equation}\label{scalar_product_discrete_interface_c}
\left<{\bf u}_{{\Gamma}2h}, {\bf v}_{{\Gamma}2h}\right>_{2hw} = 2h_1 {\bf u}^T_{{\Gamma}2h} \mathcal{W}_\Gamma^{2h}\mathcal{J}_{\Gamma\Lambda}^{2h} {\bf v}_{{\Gamma}2h},
\end{equation}
for vectors ${\bf u}_{{\Gamma}2h}$ and ${\bf v}_{{\Gamma}2h}$ of length $2n_1^{2h}$. The diagonal \revi{matrix}  $\mathcal{W}_\Gamma^{2h}$ is defined in the same way as $\mathcal{W}_\Gamma^{h}$.

Next, we present the lemma for the compatibility between the scaled interpolation and the scaled restriction operators defined in Sec.~\ref{sec_semidiscretization}.

\begin{lemma}\label{lemma1}
If the interpolation operator $P$ and restriction operator $R$ satisfy the compatibility condition \eqref{PRcomp}, 
	then   the scaled interpolation operator $\mathcal{P}$ in (\ref{scaleP}) and the scaled restriction operator $\mathcal{R}$ in (\ref{scaleR}) satisfy
	\begin{equation}\label{pr_relation}
	\left<\mathcal{P} {\bf u}_{{\Gamma}2h}, {\bf u}_{{\Gamma}h}\right>_{hw} = \left<{\bf u}_{{\Gamma}2h}, \mathcal{R}{\bf u}_{{\Gamma}h}\right>_{2hw},
	\end{equation}
for any vectors ${\bf u}_{{\Gamma}h}$ and ${\bf u}_{{\Gamma}2h}$ of length $2n_1^h$ and $2n_1^{2h}$, respectively. 
\end{lemma}
\begin{proof}
	From \eqref{scalar_product_discrete_interface_f}--\eqref{scalar_product_discrete_interface_c}, the definitions of $\mathcal{P}$ in \eqref{scaleP} and $\mathcal{R}$ in \eqref{scaleR}, we have
	\begin{align*}
	\left<\mathcal{P}{\bf u}_{{\Gamma}2h}, {\bf u}_{{\Gamma}h}\right>_{hw} &=  h_1\left[\left((\mathcal{J}_{\Gamma\Lambda}^h)^{\frac{1}{2}}({\Proj}\otimes{\bf I})(\mathcal{J}_{\Gamma\Lambda}^{2h})^{\frac{1}{2}}\right){\bf u}_{{\Gamma}2h}\right]^T\mathcal{W}_{\Gamma}^h{\bf u}_{{\Gamma}h}\\
	& = 2h_1 {\bf u}_{{\Gamma}2h}^T  \left[\left((\mathcal{J}_{\Gamma\Lambda}^{2h})^{\frac{1}{2}}\frac{1}{2}({\Proj}^T\otimes{\bf I})(\mathcal{J}_{\Gamma\Lambda}^h)^{\frac{1}{2}}\right)\mathcal{W}_{\Gamma}^h{\bf u}_{{\Gamma}h}\right]\\
	& = 2h_1{\bf u}_{{\Gamma}2h}^T  \left[\left((\mathcal{J}_{\Gamma\Lambda}^{2h})^{\frac{1}{2}}\frac{1}{2}({\Proj}^TW_\Gamma^h\otimes{\bf I})(\mathcal{J}_{\Gamma\Lambda}^h)^{\frac{1}{2}}\right){\bf u}_{{\Gamma}h}\right]\\
	& =2h_1 {\bf u}_{{\Gamma}2h}^T \left[\left((\mathcal{J}_{\Gamma\Lambda}^{2h})^{\frac{1}{2}}({W}^{2h}_\Gamma {\R}\otimes {\bf I})(\mathcal{J}_{\Gamma\Lambda}^h)^{\frac{1}{2}}\right) {\bf u}_{{\Gamma}h}\right] = \left<{\bf u}_{{\Gamma}2h}, \mathcal{R}{\bf u}_{{\Gamma}h}\right>_{2hw}.
	\end{align*}
\end{proof}

\revii{
\begin{remark}\label{remark1}
	 At every time level, the interface conditions (\ref{continuous_sol}) and (\ref{continuous_traction}) are imposed, and  Lemma \ref{lemma1} holds for the numerical solution. 
\end{remark}
}
Finally, we present an important property of the spatial difference operators. For this, we also need discrete inner products for the grid functions in $\Omega^c$ and $\Omega^f$. Let ${\bf u}_{2h}$ and ${\bf v}_{2h}$ be grid functions in the coarse domain $\Omega^c$ of length $2n_1^{2h}n_2^{2h}$, we define the two dimensional discrete inner product in $\Omega^c$ as
\begin{equation}\label{inner_product_c}
({\bf v}_{2h}, {\bf u}_{2h})_{2hw} = 4h_1h_2{\bf v}_{2h}^T\mathcal{W}^{2h}\mathcal{J}^{2h}{\bf u}_{2h},
\end{equation} 
where $\mathcal{J}^{2h} := J^{2h}\otimes{\bf I}$, $\mathcal{W}^{2h} := W^{2h}\otimes{\bf I}$ with $W^{2h} = W^{2h}_x\otimes W^{2h}_y$. The diagonal matrices $W_x^{2h}$ and $W_y^{2h}$ are the same as $W^h$ in \eqref{iph} for the $x$ and $y$ coordinate. Note that in our case, we have  $W_x^{2h} \equiv W_\Gamma^{2h}$. The corresponding discrete inner product in $\Omega^f$ is defined analogously as
\begin{equation}\label{inner_product_f}
({\bf v}_h, {\bf u}_h)_{hw} = h_1h_2{\bf v}_h^T\mathcal{W}^h\mathcal{J}^h{\bf u}_h.
\end{equation} 
Using the SBP properties \eqref{sbp_dx}, \eqref{sbp_dxx_gp} and \eqref{sbp_dxx}, by a similar derivation as in \cite{petersson2015wave, Zhang2021}, we have
\begin{equation}\label{sbp_c}
\left({\bf v}_{2h}, (\mathcal{J}^{2h})^{-1}\wt{\mathcal{L}}^{2h}{\bf u}_{2h}\right)_{2hw} = -\mathcal{S}_{2h}({\bf v}_{2h}, {\bf u}_{2h}) + B_{2h}({\bf v}_{2h}, {\bf u}_{2h}),
\end{equation}
where $\mathcal{S}_{2h}({\bf v}_{2h}, {\bf u}_{2h})$ is a symmetric, positive semi-definite bilinear form given in Appendix \ref{definitions}. The term corresponding to the interface contribution is 
\begin{equation}\label{B2h}
B_{2h}({\bf v}_{2h}, {\bf u}_{2h}) = 2h_1\sum_{{\bf i}\in I_{\Gamma^c}} (\mathcal{W}^{2h}{\bf v}_{2h})_{\bf i}\cdot\left(\wt{\mathcal{A}}_2^{2h}{\bf u}_{2h}\right)_{\bf i} \revii{ = \left<{\bf v}_{2h}, \left(\mathcal{J}_{\Gamma\Lambda}^{2h}\right)^{-1}\widetilde{\mathcal{A}}_2^{2h}{\bf u}_{2h}\right>_{2hw}}.
\end{equation}
Correspondingly, in the coarse domain we have
\begin{equation}\label{sbp_f}
\left({\bf v}_h, (\mathcal{J}^h)^{-1}\hat{\mathcal{L}}^{h}{\bf u}_h\right)_{hw} = -\mathcal{S}_{h}({\bf v}_h, {\bf u}_h) + B_{h}({\bf v}_h, {\bf u}_h) + h_2\omega_1\sum_{{\bf i}\in\Gamma^f} (\mathcal{W}^h{\bf v}_h)_{\bf i}\cdot{\bm \eta}_{\bf i},
\end{equation}
where the interface contribution is 
\begin{equation}\label{Bh}
B_{h}({\bf v}_h, {\bf u}_h) = -h_1\sum_{{\bf i}\in I_{\Gamma^f}} (\mathcal{W}^h{\bf v}_h)_{\bf i}\cdot\left({\mathcal{A}}_2^{h}{\bf u}_h\right)_{\bf i} \revii{= -\left<{\bf v}_{h}, \left(\mathcal{J}_{\Gamma\Lambda}^{2h}\right)^{-1}{\mathcal{A}}_2^{h}{\bf u}_{h}\right>_{hw}},
\end{equation}
and $\mathcal{S}_{h}({\bf v}_h, {\bf u}_h)$ is also a symmetric, positive semi-definite bilinear form given in Appendix \ref{definitions}. Note that in \eqref{sbp_c} and (\ref{sbp_f}), we have omitted terms corresponding to physical boundary conditions. 

For the symmetric, positive semi-definite bilinear forms $\mathcal{S}_h$ and $\mathcal{S}_{2h}$, there exist square matrices $K_{h}$ and $K_{2h}$ such that
\begin{equation}\label{K}
\mathcal{S}_h({\bf v}_h, {\bf u}_h) = \left({\bf v}_h, \varrho^hK_h{\bf u}_h\right)_{hw} \ \ \mbox{and}\ \ \mathcal{S}_{2h}({\bf v}_{2h}, {\bf u}_{2h}) = \left({\bf v}_{2h}, \varrho^{2h}K_{2h}{\bf u}_{2h}\right)_{2hw},
\end{equation}
where $\varrho^h = \rho^h\otimes{\bf I}$ and $\varrho^{2h} = \rho^{2h}\otimes{\bf I}$. Since the above inner products are weighted, the matrices $K_h$ and $K_{2h}$ are not symmetric, but they satisfy the relations
\begin{align}
\left({\bf v}_h, \varrho^hK_h{\bf u}_h\right)_{hw} = \left({\bf u}_h, \varrho^hK_h{\bf v}_h\right)_{hw} ,\ \ \left({\bf v}_h, \varrho^hK_h{\bf v}_h\right)_{hw} \geq 0,\label{kh}\\
\left({\bf v}_{2h}, \varrho^{2h}K_{2h}{\bf u}_{2h}\right)_{2hw} = ({\bf u}_{2h}, \varrho^{2h}K_{2h}{\bf v}_{2h})_{2hw},\ \ \left({\bf v}_{2h}, \varrho^{2h}K_{2h}{\bf v}_{2h}\right)_{2hw} \geq 0.\label{k2h}
\end{align}
In addition, the scaled matrices 
\begin{equation}\label{kbar}
\bar{K}_h = (\mathcal{W}_\rho^h)^{1/2}K_h(\mathcal{W}_\rho^h)^{-1/2},\ \ \bar{K}_{2h} = (\mathcal{W}_\rho^{2h})^{1/2}K_{2h}(\mathcal{W}_\rho^{2h})^{-1/2},
\end{equation}
are symmetic,  
where $\mathcal{W}_\rho^h := (\rho^h W^h)\otimes{\bf I}$ and $\mathcal{W}_\rho^{2h} := (\rho^{2h} W^{2h})\otimes{\bf I}$. The scaled matrices are also positive semi-definite  because of the relations
\begin{align*}
\left({\bf u}_h, \varrho^hK_h{\bf v}_h\right)_{hw} &= \left((\mathcal{W}_\rho^h)^{1/2}{\bf u}_h, \bar{K}^h(\mathcal{W}_\rho^h)^{1/2}{\bf v}_h\right)_h,\\ 
\left({\bf u}_{2h}, \varrho^{2h} K_{2h} {\bf v}_{2h}\right)_{2hw} &= \left((\mathcal{W}_\rho^{2h})^{1/2}{\bf u}_{2h}, \bar{K}^{2h}(\mathcal{W}_\rho^{2h})^{1/2}{\bf v}_{2h}\right)_{2h},
\end{align*}
where $(\cdot,\cdot)_h$ and $(\cdot,\cdot)_{2h}$ are the standard discrete $L_2$ inner product. 

\subsubsection{Energy conservation} In this subsection, we present the main result on fully discrete stability. We first introduce fully discrete energy and prove that it is always nonnegative under a CFL condition on the time step. After that, we prove that the fully discrete energy is conserved.

To clearly present the contents, we introduce the difference operators for the first and second temporal derivatives
\[D_t {\bf c}^n := \frac{{\bf c}^{n+1} - {\bf c}^n}{\Delta_t}, \ D_t^2 {\bf c}^n := \frac{{\bf c}^{n+1} - 2{\bf c}^n + {\bf c}^{n-1}}{\Delta_t^2}. \]
We also use the convention
\begin{equation}\label{Dt2}
D_t^2 {\bf c}^{\ast, n} := \frac{{\bf c}^{\ast, n+1} - 2{\bf c}^n + {\bf c}^{n-1}}{\Delta_t^2},
\end{equation}
and the notation
\[\delta {\bf c}^n := {\bf c}^{n+1} - {\bf c}^{n-1}.\]
The operators $D_t$, $D_t^2$ and $\delta$ can be applied to ${\bf f}$ in the same way. The following theorem defines the fully discrete energy. 

\begin{theorem}\label{energy_pos}
	The fully discrete energy defined by
	\begin{multline}\label{discrete_energy}
	E^{n+1/2} =	\left((\varrho^h)^{\frac{1}{2}}D_t{\bf f}^n, (\varrho^h)^{\frac{1}{2}}D_t{\bf f}^n\right)_{hw}+ \left((\varrho^{2h})^{\frac{1}{2}}D_t{\bf c}^n, (\varrho^{2h})^{\frac{1}{2}}D_t{\bf c}^n\right)_{2hw} +  \mathcal{S}_{2h}({\bf c}^{n+1}\!, {\bf c}^{n}\!)\\+ \mathcal{S}_h({\bf f}^{n+1}, {\bf f}^{n})  - \frac{\Delta_t^2}{12}\left((\mathcal{J}^h)^{-1}\hat{\mathcal{L}}^h{\bf f}^{n+1}, (\mathcal{J}_\rho^h)^{-1}\hat{\mathcal{L}}^h{\bf f}^{n}\right)_{hw} 
	\!\!- \frac{\Delta_t^2}{12}\left(\!(\mathcal{J}^{2h})^{-1}\wt{\mathcal{L}}^{2h}{\bf c}^{n+1}, (\mathcal{J}_\rho^{2h})^{-1}\wt{\mathcal{L}}^{2h}{\bf c}^{n}\!\right)_{2hw}.
	\end{multline}
	 is nonnegative under the time step restriction 
	 \begin{equation}\label{timesteprestriction}
	 \Delta_t\leq \min\left\{\frac{2\sqrt{3}}{\max_j\sqrt{\kappa_j^h}},
	 \frac{2\sqrt{3}}{\max_j\sqrt{\kappa_j^{2h}}}\right\},
	 \end{equation}
	 where $\kappa_j^h$ and $\kappa_j^{2h}$ are the $j$-th eigenvalue of $\bar{K}_h$ and $\bar{K}_{2h}$ defined in (\ref{kbar}), respectively.
\end{theorem}
\begin{proof}
	The proof is given in Appendix \ref{positive_energy}.
\end{proof}

It is important to note that the time step restriction is determined by the spectral radius of $\bar{K}_h$ and $\bar{K}_{2h}$, which are related to the bilinear forms $\mathcal{S}_{h}$ and $\mathcal{S}_{2h}$ in Appendix \ref{definitions}, but not related to the interpolation and restriction operators used for imposing the numerical interface conditions. This fact leads to a favorable time step restriction.


\begin{theorem}\label{energy_conserve_th}
	Consider the full discretization in Sec.~\ref{time_integrator} and the fully discrete energy in (\ref{discrete_energy}). Under the time step restriction \eqref{timesteprestriction}, we have the energy conservation
	\[E^{n+1/2} = E^{n-1/2}.\]
\end{theorem}
\begin{proof}
	 From the definition (\ref{discrete_energy}), we obtain
	\begin{multline}\label{difference_energy}
	E^{n+1/2} - E^{n-1/2} 
= \left(\delta {\bf f}^n, \varrho^hD_t^2{\bf f}^n\right)_{hw} + \left(\delta {\bf c}^n, \varrho^{2h}D_t^2{\bf c}^n\right)_{2hw} + \mathcal{S}_h(\delta {\bf f}^n, {\bf f}^n)+ \mathcal{S}_{2h}(\delta {\bf c}^n, {\bf c}^n)\\
	- \frac{\Delta_t^2}{12}\left((\mathcal{J}^h)^{-1}\hat{\mathcal{L}}^h(\delta {\bf f}^n), (\mathcal{J}_\rho^h)^{-1}\hat{\mathcal{L}}^h{\bf f}^{n}\right)_{hw}
	- \frac{\Delta_t^2}{12}\left(\!(\mathcal{J}^{2h})^{-1}\wt{\mathcal{L}}^{2h}(\delta {\bf c}^n), (\mathcal{J}_\rho^{2h})^{-1}\wt{\mathcal{L}}^{2h}{\bf c}^{n}\!\right)_{2hw}.
	\end{multline} 
   We consider term by term in \eqref{difference_energy}. In the first two terms on the right hand side of (\ref{difference_energy}), \revii{ we note that (\ref{predictor_c}) and (\ref{predictor_f}) are equivalent to } 
	\begin{equation}\label{predictor_stable}
	D_t^2 {\bf c}_{\bf i}^{\ast, n} = (\mathcal{J}_\rho^{2h})^{-1}\wt{\mathcal{L}}_{\bf i}{\bf c}^n,\ \ {\bf i}\in I_{\Omega^c}\ \ \  \mbox{and}\ \ \ D_t^2 {\bf f}_{\bf i}^{\ast, n} = (\mathcal{J}_\rho^h)^{-1}\hat{\mathcal{L}}_{\bf i}{\bf f}^n,\ \ {\bf i}\in I_{\Omega^f},
	\end{equation}
\revii{based on the definition of $D_t^2$ in (\ref{Dt2})}. By the  first equality in (\ref{predictor_stable}) and the scheme (\ref{corrector_c}), we obtain the inner product 
	\begin{multline}\label{corrector_stable_c1}
	(\delta {\bf c}^n, {\bf c}^{n+1})_{2hw} = (\delta {\bf c}^n, {\bf c}^{\ast,n+1})_{2hw} 
	+ \frac{\Delta_t^4}{12} \left(\delta {\bf c}^n, (\mathcal{J}_\rho^{2h})^{-1}\wt{\mathcal{L}}^{2h}\left(D_t^2 {\bf c}^{\ast,n}\right)\right)_{2hw}\\
	=  \big(\delta {\bf c}^n, 2{\bf c}^{n} - {\bf c}^{n-1}\big)_{2hw}+ \big(\delta {\bf c}^n,\Delta_t^2(\mathcal{J}_\rho^{2h})^{-1}\wt{\mathcal{L}}^{2h}{\bf c}^n\big)_{2hw} 
	+ \frac{\Delta_t^4}{12} \left(\delta {\bf c}^n, (\mathcal{J}_\rho^{2h})^{-1}\wt{\mathcal{L}}^{2h}\left(D_t^2 {\bf c}^{\ast,n}\right)\right)_{2hw}.
	\end{multline}
\revi{ Multiplying (\ref{corrector_stable_c1}) by the diagonal matrix $\varrho^{2h}/\Delta_t^2 = (\rho^{2h}\otimes{\bf I})/\Delta_t^2$ and rearranging the terms in the resulting equality}, we have
	\begin{equation}\label{corrector_stable_c2}
	\left(\delta {\bf c}^n, \varrho^{2h}D_t^2 {\bf c}^n\right)_{2hw} = \left(\delta {\bf c}^n, (\mathcal{J}^{2h})^{-1}\wt{\mathcal{L}}^{2h}{\bf c}^n\right)_{2hw} 
	+ \frac{\Delta_t^2}{12}\left(\delta {\bf c}^n, (\mathcal{J}^{2h})^{-1}\wt{\mathcal{L}}^{2h}\left(D_t^2 {\bf c}^{\ast,n}\right)\right)_{2hw}.
	\end{equation}
	Using (\ref{sbp_c}), the last term in (\ref{corrector_stable_c2}) can be simplified to 
	\begin{equation}\label{corrector_stable_c3}
	\left(\delta {\bf c}^n, (\mathcal{J}^{2h})^{-1}\wt{\mathcal{L}}^{2h}\left(D_t^2 {\bf c}^{\ast,n}\right)\right)_{2hw}
	= -\mathcal{S}_{2h}\left(\delta {\bf c}^n, D_t^2 {\bf c}^{\ast,n}\right) + B_{2h}\left(\delta {\bf c}^n, D_t^2 {\bf c}^{\ast,n}\right).
	\end{equation}

	Similarly, by the scheme  (\ref{corrector_f}) and  the second equality in (\ref{predictor_stable}), we have
	\begin{equation}\label{corrector_stable_f2}
	\left(\delta {\bf f}^n, \varrho^hD_t^2 {\bf f}^n\right)_{hw} = \left(\delta {\bf f}^n, (\mathcal{J}^h)^{-1}\hat{\mathcal{L}}^{h}{\bf f}^n\right)_{hw}
	+ \frac{\Delta_t^2}{12}\left(\delta {\bf f}^n, (\mathcal{J}^h)^{-1}\hat{\mathcal{L}}^{h}\left(D_t^2 {\bf f}^{\ast,n}\right)\right)_{hw}.
	\end{equation}
	By (\ref{sbp_f}), the last term in (\ref{corrector_stable_f2}) can be written as 
	\begin{equation}\label{corrector_stable_f3}
	\left(\delta {\bf f}^n\!, (\mathcal{J}^h)^{-1}\hat{\mathcal{L}}^{h}\left(D_t^2 {\bf f}^{\ast,n}\right)\right)_{hw}
	\!\!\!\!\!\!= -\mathcal{S}_{h}\left(\delta {\bf f}^n\!, D_t^2 {\bf f}^{\ast,n}\right) \!+ B_{h}\left(\delta {\bf f}^n,\! D_t^2 {\bf f}^{\ast,n}\right) 
	\!+\! h_2\omega_1\!\!\sum_{{\bf i}\in I_{\Gamma^f}}\!\!\left(\mathcal{W}^h\delta {\bf f}^n\right)_{\bf i}\cdot\left(D_t^2 {\bm \eta}_{\bf i}^{\ast,n}\right).
	\end{equation}
	Substituting (\ref{corrector_stable_c3}) into (\ref{corrector_stable_c2}), and (\ref{corrector_stable_f3}) into (\ref{corrector_stable_f2}), yields
		\begin{align}
	\left(\delta {\bf c}^n, \varrho^{2h}D_t^2{\bf c}^n \right)_{2hw} &= (\delta {\bf c}^n,  (\mathcal{J}^{2h})^{-1}\wt{\mathcal{L}}^{2h}{\bf c}^n)_{2hw} 
	-\frac{\Delta_t^2}{12}\mathcal{S}_{2h}\left(\delta {\bf c}^n, D_t^2{\bf c}^{\ast,n}\right) +\frac{\Delta_t^2}{12}B_{2h}\left(\delta {\bf c}^n,D_t^2{\bf c}^{\ast,n} \right),\label{stability_eq1}\\
	\left(\delta {\bf f}^n, \varrho^hD_t^2{\bf f}^n\right)_{hw} &= (\delta {\bf f}^n,  (\mathcal{J}^h)^{-1}\hat{\mathcal{L}}^{2h}{\bf f}^n)_{hw} 
	-\frac{\Delta_t^2}{12}\mathcal{S}_{h}\left(\delta {\bf f}^n, D_t^2{\bf f}^{\ast,n}\right)\label{stability_eq2}\\
	&\quad+\frac{\Delta_t^2}{12}B_{h}\left(\delta {\bf f},D_t^2{\bf f}^{\ast,n}\right) 
	+ \frac{\Delta_t^2}{12}h_2\omega_1\sum_{{\bf i}\in I_{\Gamma^f}}\big(\mathcal{W}^h\delta {\bf f}^n\big)_{\bf i}\cdot\left(D_t^2{\bm \eta}_{\bf i}^{\ast, n}\right).\notag
		\end{align}
	For the third and the fourth terms in (\ref{difference_energy}), by the relation (\ref{sbp_c}) and (\ref{sbp_f}), we rewrite the bilinear forms as
	\begin{align}
		\mathcal{S}_{2h}(\delta {\bf c}^n, {\bf c}^n) &= -\left(\delta {\bf c}^n, (\mathcal{J}^{2h})^{-1}\wt{\mathcal{L}}^{2h}{\bf c}^n\right)_{2hw} + B_{2h}(\delta {\bf c}^n, {\bf c}^n),\label{energy_stable_c1} \\
		\mathcal{S}_h(\delta {\bf f}^n, {\bf f}^n) &= -\left(\delta {\bf f}^n, (\mathcal{J}^h)^{-1}\hat{\mathcal{L}}^h{\bf f}^n\right)_{hw} + B_h(\delta {\bf f}^n, {\bf f}^n) 
	+ h_2\omega_1\sum_{{\bf i}\in\Gamma^f}\left(\mathcal{W}^h(\delta {\bf f}^n)\right)_{\bf i}\cdot {\bm \eta}_{\bf i}^n. \label{energy_stable_f1}
	\end{align}
    Finally, we consider the last two terms in (\ref{difference_energy}). Taking ${\bf v}_h$ to be $(\mathcal{J}_\rho^h)^{-1}\hat{\mathcal{L}}^h{\bf f}^n$, ${\bf u}_h$ to be $\delta {\bf f}^n$ in (\ref{sbp_f}), ${\bf v}_{2h}$ to be $(\mathcal{J}_\rho^{2h})^{-1}\wt{\mathcal{L}}^{2h}{\bf c}^n$, ${\bf u}_{2h}$ to be $\delta {\bf c}^n$ in (\ref{sbp_c}), and then using  (\ref{predictor_c}), (\ref{predictor_f}), and the symmetric property of $\mathcal{S}_h$ and $\mathcal{S}_{2h}$, we obtain 
\begin{align}
	-\frac{\Delta_t^2}{12}\left((\mathcal{J}^h)^{-1}\hat{\mathcal{L}}^h(\delta {\bf f}^n), (\mathcal{J}_\rho^h)^{-1}\hat{\mathcal{L}}^h{\bf f}^{n}\right)_{hw} 
	&= \frac{\Delta_t^2}{12}\mathcal{S}_h\left(\delta {\bf f}^n, D_t^2{\bf f}^{\ast,n}\right)  -\frac{\Delta_t^2}{12} B_{h}\left(D_t^2{\bf f}^{\ast,n}, \delta {\bf f}^n\right)\label{energy_stable_f2}\\
	&\quad -\frac{\Delta_t^2}{12} h_2\omega_1\sum_{{\bf i}\in\Gamma^{f}}\left(\mathcal{W}^hD_t^2{\bf f}^{\ast,n}\right)_{\bf i}\cdot (\delta {\bm \eta}_{\bf i}^n),\notag\\
	-\frac{\Delta_t^2}{12}\left((\mathcal{J}^{2h})^{-1}\wt{\mathcal{L}}^{2h}(\delta{\bf c}^n), (\mathcal{J}_\rho^{2h})^{-1}\wt{\mathcal{L}}^{2h}{\bf c}^{n}\right)_{2hw} 
	&= \frac{\Delta_t^2}{12}\mathcal{S}_{2h}\left(\delta {\bf c}^n, D_t^2{\bf c}^{\ast,n}\right) -\frac{\Delta_t^2}{12} B_{2h}\left(D_t^2{\bf c}^{\ast,n}, \delta {\bf c}^n\right).		\label{energy_stable_c2}
	\end{align}
	Substituting (\ref{stability_eq1}) -- (\ref{energy_stable_c2}) into (\ref{difference_energy}), we have
	\begin{equation}
	E^{n+1/2} - E^{n-1/2} = \sigma_1 + \frac{\Delta_t^2}{12}\sigma_2 - \frac{\Delta_t^2}{12}\sigma_3, 
	\end{equation}
	where
\begin{align*}
	\sigma_1 &= B_{2h}(\delta {\bf c}^n, {\bf c}^n) + B_h(\delta {\bf f}^n, {\bf f}^n) + h_2\omega_1\sum_{{\bf i}\in \Gamma^f}\left(\mathcal{W}^h\delta {\bf f}^n\right)_{\bf i}\cdot{\bm \eta}^n_{\bf i},\\
\sigma_2 &= B_{2h}\left(\delta {\bf c}^n, D_t^2{\bf c}^{\ast,n}\right) + B_h\left(\delta {\bf f}^n, D_t^2{\bf c}^{\ast,n}\right) 
+ h_2\omega_1\sum_{{\bf i}\in \Gamma^f}\left(\mathcal{W}^h\delta {\bf f}^n\right)_{\bf i}\cdot\left(D_t^2{\bm \eta}_{\bf i}^{\ast,n}\right),\\
\sigma_3 &=  B_{2h}\left(D_t^2{\bf c}^{\ast,n}, \delta {\bf c}^n\right) + B_h\left( D_t^2{\bf f}^{\ast,n}, \delta {\bf f}^n\right) 
+ h_2\omega_1\sum_{{\bf i}\in \Gamma^f}\left(\mathcal{W}^hD_t^2{\bf f}^{\ast,n}\right)_{\bf i}\cdot (\delta {\bm \eta}_{\bf i}^n).
\end{align*}
For $\sigma_1$: we combine the expressions of $B_h$ in (\ref{Bh}), $B_{2h}$ in (\ref{B2h}). Using the inner products  (\ref{scalar_product_discrete_interface_f})--(\ref{scalar_product_discrete_interface_c}), and the quatities $\{{\bf c}^{n-1}, {\bf f}^{n-1}\}, \{{\bf c}^{n}, {\bf f}^{n}\}, \{{\bf c}^{n+1}, {\bf f}^{n+1}\}$ satisfying the interface conditions (\ref{continuous_sol}) and (\ref{continuous_traction}), and Lemma \ref{lemma1}, we have
\begin{align*}
\sigma_1 &= \left<\delta{\bf f}^n, (\mathcal{J}_{\Gamma\Lambda}^h)^{-1}(-\mathcal{A}_2^h{\bf f}^n + h_2\omega_1{\bm \eta}^n)\right>_{hw} + \left<\delta {\bf c}^n, (\mathcal{J}_{\Gamma\Lambda}^{2h})^{-1}\wt{\mathcal{A}}_2^{2h}{\bf c}^n\right>_{2hw}\\
&= \left<\mathcal{P}(\delta {\bf c}^n), (\mathcal{J}_{\Gamma\Lambda}^h)^{-1}(-\mathcal{A}_2^h{\bf f}^n + h_2\omega_1{\bm \eta}^n)\right>_{hw} + \left<\delta {\bf c}^n, (\mathcal{J}_{\Gamma\Lambda}^{2h})^{-1}\wt{\mathcal{A}}_2^{2h}{\bf c}^n\right>_{2hw}\\
&= \left<\delta {\bf c}^n, \mathcal{R}\left((\mathcal{J}_{\Gamma\Lambda}^h)^{-1}(-\mathcal{A}_2^h{\bf f}^n + h_2\omega_1{\bm \eta}^n)\right)\right>_{2hw}+ \left<\delta {\bf c}^n, (\mathcal{J}_{\Gamma\Lambda}^{2h})^{-1}\wt{\mathcal{A}}_2^{2h}{\bf c}^n\right>_{2hw} = 0.
\end{align*}


By similar analysis as for $\sigma_1$, and with the fact $\{{\bf c}^{\ast, n+1}, {\bf f}^{\ast, n+1}\}$ satisfying the interface condition (\ref{continuous_sol}) and (\ref{continuous_traction}), we also have $\sigma_2 = 0$ and $\sigma_3 = 0$. This concludes the proof of $E^{n+1/2} = E^{n-1/2}$.
\end{proof}

\section{Numerical experiments}\label{numerical_experiments}
We present four numerical experiments in this section.  In Sec.~\ref{sec_manufactured_sol}, we verify the convergence order of the proposed fourth and sixth order methods by a manufactured solution. In Sec.~\ref{sec:stoneley_wave}, we present the error history and convergence rate of the Stoneley surface wave. Next, the energy conservation property with heterogeneous and discontinuous material properties is investigated in Sec.~\ref{sec:energy_conservation}. Finally, in Sec.~\ref{cfl_sat}, we compare the CFL number of the proposed methods and SBP-SAT methods. \revii{In all numerical experiments, Dirichlet boundary conditions are imposed by the injection method, and free surface boundary conditions are imposed by the ghost point method, see \cite{petersson2015wave} for the details.}

As shown in Theorem \ref{energy_pos}, the precise time step restriction requires the computation of the spectral radius of $\bar{K}_h$ and $\bar{K}_{2h}$ defined in (\ref{kbar}). Due to the boundary modification in the SBP operators, a closed form expression for the spectral radius does not exist and its numerical computation is expensive. 
\revii{To this end, we approximate the theoretical time step \eqref{timesteprestriction} by}
\begin{equation}\label{time_step}
\Delta_t \leq C_{\mbox{cfl}} \min\left\{\min\{h_1, h_2\}{\Big/}\sqrt{\zeta^h_{\mbox{max}}}, \min\{2h_1, 2h_2\}{\Big/}\sqrt{\zeta^{2h}_{\mbox{max}}}\right\},
\end{equation}
where $\zeta^h_{\mbox{max}}$ and $\zeta^{2h}_{\mbox{max}}$ are the  maximum eigenvalues of the $2\times 2$ matrices
\begin{equation*}
T_{\bf i}^{f} = \frac{1}{\rho^{f}(\bf r_i)}\begin{pmatrix}
\mbox{Tr}(N_{11}^{f}(\bf r_i)) & \mbox{Tr}(N_{12}^{f}(\bf r_i))\\
\mbox{Tr}(N_{21}^{f}(\bf r_i)) & \mbox{Tr}(N_{22}^{f}(\bf r_i))
\end{pmatrix},{\bf i}\in I_{\Omega^f},\  T_{\bf i}^{c} = \frac{1}{\rho^{c}(\bf r_i)}\begin{pmatrix}
\mbox{Tr}(N_{11}^{c}(\bf r_i)) & \mbox{Tr}(N_{12}^{c}(\bf r_i))\\
\mbox{Tr}(N_{21}^{c}(\bf r_i)) & \mbox{Tr}(N_{22}^{c}(\bf r_i))
\end{pmatrix},{\bf i}\in I_{\Omega^c},
\end{equation*}
respectively. Here, $\mbox{Tr}(N_{lm}^{f}(\bf r_i))$ and $\mbox{Tr}(N_{lm}^{c}(\bf r_i))$ represent the trace of $2\times2$ matrices $N_{lm}^{f}(\bf r_i)$ and $N_{lm}^{c}(\bf r_i)$, respectively. \revii{The Courant number is set to be $C_{\mbox{cfl}} = 1.3$, which amounts to be 10\% smaller than for the spatially periodic case}, see details in  \cite{petersson2015wave, sjogreen2012fourth}. 


\subsection{Curvilinear grids}\label{sec_manufactured_sol}
We use the method of the manufactured solution to verify the convergence rate of the proposed schemes. We choose the mapping of the coarse domain $\Omega^c$ and fine domain $\Omega^f$ as
\begin{equation*}
{\bf X}^c({\bf r}) = \begin{pmatrix}
2\pi r\\
s\theta_i(r) + (1 - s)\theta_b(r)
\end{pmatrix},\ \ {\bf X}^f({\bf r}) = \begin{pmatrix}
2\pi r\\
s\theta_t(r) + (1 - s)\theta_i(r)
\end{pmatrix},
\end{equation*}
respectively. Here, $0\leq r, s \leq 1$; $\theta_i$, $\theta_b$ and $\theta_t$ represent the interface surface geometry, bottom surface geometry and top surface geometry, respectively. Precisely,
\begin{equation*}
\theta_i(r) \!=\! \pi + 0.2\sin(4\pi r), \ \theta_b(r) \!=\! 0.2\exp\left(\!-\frac{(r - 0.6)^2}{0.04}\!\right),\  \theta_t(r) \!=\! 2\pi + 0.2\exp\left(\!-\frac{(r - 0.5)^2}{0.04}\!\right).
\end{equation*}
In the entire domain, we choose the density 
$\rho(x, y) = 2 + \sin(x + 0.3)\sin(y - 0.2),$
and material parameters
\begin{equation*}
\mu(x, y) = 3 + \sin(3x + 0.1)\sin(y),\ \ \ 
\lambda(x, y) = 21 + \cos(x + 0.1)\sin^2(3y).
\end{equation*}
In addition, we impose a boundary forcing on the top surface and Dirichlet boundary conditions for the other boundaries. The external forcing, top boundary forcing and initial conditions are chosen such that the solutions for both fine domain ($\bf F$) and coarse domain ($\bf C$) are ${\bf F}(\cdot, t) = {\bf C}(\cdot, t) = {\bf u}(\cdot, t) = (u_1 (\cdot, t), u_2(\cdot, t))^T$ with
\[u_1(\cdot,t) = \cos(x + 0.3)\sin(y + 0.2)\cos(t^2)\ \ \mbox{and}\ \  u_2(\cdot,t) = \sin(x + 0.2)\cos(y + 0.2)\sin(t).\]

The problem is evolved until final time $T = 1$ with $2n_1^h = n_1^{2h}+1$, $2n_2^h+1 = n_1^h$, $2n_2^{2h}+1 = n_1^{2h}$. In Table \ref{tab_curvi_interface}, we present the \revi{errors in the discrete $l^2$ norm} and the corresponding convergence rate with both fourth order and sixth order methods. We observe that the convergence rate is fourth order for the fourth order method and fifth order for the sixth order method. Even though the boundary accuracy of the fourth order SBP operator is only second order, the optimal convergence rate is fourth order; the boundary accuracy of the sixth order method only third order, the optimal convergence rate is fifth order. For a more detailed analysis of the convergence rate, we refer to \cite{Wang2018b}.

\begin{table}
	\centering
	\begin{tabular}{lllll}
		\hline
		$n_1^{2h}$ & $l^2$ error (4th) & convergence rate & $l^2$ error (6th) & convergence rate \\
		\hline
		61 & $7.5505\times 10^{-5}$ & &  $2.2767\times 10^{-5}$ &\\
		121 & $4.4768\times 10^{-6}$ & 4.0760 & $4.3716\times 10^{-7}$& 5.7026\\
		241 & $2.6793\times 10^{-7}$ &4.0625 & $1.0207\times 10^{-8}$  & 5.4205 \\
		481 &$1.6393\times 10^{-8}$  &4.0307 & $2.7869\times 10^{-10}$ & 5.1948 \\
      961 & $1.0176\times 10^{-9}$  &4.0101 & $8.2136\times 10^{-12}$ & 5.0845 \\
		\hline
	\end{tabular}
	\caption{$l^2$ errors and convergence rates for the manufactured solution with a smooth curved interface.}
	\label{tab_curvi_interface}
\end{table}

\subsection{Stoneley wave}\label{sec:stoneley_wave}
A Stoneley interface wave clings to the flat interface between two elastic media and decays exponentially into the media, away from the interface is investigated in this section. Define the p-wave speed and s-wave speed in the fine domain $\Omega^f$ and the coarse domain $\Omega^c$ by
\[\alpha^f = \sqrt{\left({\lambda^f+\mu^f}\right)/{\rho^f}},\ \ \ \beta^f = \sqrt{{\mu^f}/{\rho^f}};\ \ \ \alpha^c = \sqrt{\left({\lambda^c+\mu^c}\right)/{\rho^c}},\ \ \ \beta^c = \sqrt{{\mu^c}/{\rho^c}},\]
respectively. We consider two elastic half-planes with a flat interface at $y = 0$. In particular, we consider the mapping of the coarse domain $\Omega^c$ and fine domain $\Omega^f$ as
\begin{equation*}
{\bf X}^c({\bf r}) = \begin{pmatrix}
2\pi r\\
-4\pi(1 - s)
\end{pmatrix},\ \ \  {\bf X}^f({\bf r}) = \begin{pmatrix}
2\pi r\\
4\pi s
\end{pmatrix},
\end{equation*}
respectively. Here, $0\leq r, s \leq 1$.
In addition, we assume that the Stoneley wave is $2\pi$-periodic in the $x$-direction. The component of the Stoneley wave in the half-plane $y\geq 0$ (fine domain $\Omega^f$) can then be written as
\begin{multline}\label{stoneley_f}
\begin{pmatrix}
u_1^f(\cdot,t)\\
u_2^f(\cdot,t)
\end{pmatrix} = A^f\exp(-y\sqrt{1-(c_s/\alpha^f)^2}) \begin{pmatrix}
\cos(x - c_st) \\
-\sqrt{1 - (c_s/\alpha^f)^2}\sin(x - c_s t) 
\end{pmatrix}\\
+ B^f\exp(-y\sqrt{1 - (c_s/\beta^f)^2})\begin{pmatrix}
-\sqrt{1 - (c_s/\beta^f)^2}\cos(x - c_st)\\
\sin(x- c_st)
\end{pmatrix},
\end{multline}
and the component of the Stoneley wave in the half-plane $y \leq 0$ (coarse domain $\Omega^c$) is written as
\begin{multline}\label{stoneley_c}
\begin{pmatrix}
u_1^c(\cdot,t)\\
u_2^c(\cdot,t)
\end{pmatrix} = A^c\exp(y\sqrt{1-(c_s/\alpha^c)^2}) \begin{pmatrix}
\cos(x - c_st) \\
\sqrt{1 - (c_s/\alpha^c)^2}\sin(x - c_s t) 
\end{pmatrix}\\
+ B^c\exp(y\sqrt{1 - (c_s/\beta^c)^2})\begin{pmatrix}
\sqrt{1 - (c_s/\beta^c)^2}\cos(x - c_st)\\
\sin(x - c_st)
\end{pmatrix}.
\end{multline}
To get the coefficients $A^f$, $B^f$, $A^c$, $B^c$ and the phase velocity $c_s$ in (\ref{stoneley_f}) and (\ref{stoneley_c}), we apply the interface condition (\ref{interface_cond}) which leads to the following linear system
\begin{equation}\label{Stoneley_system}
D(A^f, B^f, A^c, B^c)^T = (0, 0, 0, 0)^T
\end{equation}
where 
\begin{equation*}
\!\!\!D \!= \!\begin{pmatrix}
1 & -\eta^f & -1 & -\eta^c \\
\gamma^f & -1 & \gamma^c& 1\\
-\rho^fc_s^2+2\rho^f(\beta^f)^2 &-2\rho^f(\beta^f)^2\eta^f& 
\rho^cc_s^2-2\rho^c(\beta^c)^2& -2\rho^c(\beta^c)^2\eta^c\\
-2\rho^f(\beta^f)^2\gamma^f& \rho^f(\beta^f)^2(2-(\frac{c_s}{\beta^f})^2) & 
-2\rho^c(\beta^c)^2\gamma^c& -\rho^c(\beta^c)^2(2-(\frac{c_s}{\beta^c})^2)
\end{pmatrix},
\end{equation*}
with
\[\eta^c = \sqrt{1 - \left(c_s/\beta^c\right)^2}, \ \eta^f = \sqrt{1 - \left({c_s}/{\beta^f}\right)^2},\ \gamma^c = \sqrt{1 - \left({c_s}/{\alpha^c}\right)^2},\ \gamma^f = \sqrt{1 - \left({c_s}/{\alpha^f}\right)^2}.\]
First, note that (\ref{Stoneley_system}) is an under-determined linear system, we have infinitely many solutions; second, to solve (\ref{Stoneley_system}), we must have $\det(D) = 0$, which leads to a value for the phase velocity $c_s$. To uniquely determine the values for $A^f, B^f, A^c$ and $B^c$, we impose $A^f = 1$ for all simulations of this section.

\begin{table}
	\centering
	\begin{tabular}{cccc}
		\hline
		$(\mu^f = \mu, \lambda^f, \rho^f)$ & $(\mu^c = 2\mu, \lambda^c, \rho^c)$   & $c_s$ & $2\pi/c_s$\\
		\hline
		$(1, 1, 2)$ & $ (2, 2, 1.9999)$ & $0.995069948673601$ & $6.31$ \\
		$(0.1, 1, 2)$ & $ (0.2, 2, 1.9999)$ & $0.315111729874378$ & $19.93$ \\
		$(0.01, 1, 2)$ & $ (0.02, 2, 1.9999)$ & $0.099674435254786$ & $63.03$ \\
		$(0.001, 1, 2)$ & $ (0.002, 2, 1.9999)$ & $ 0.031520771980397$ & $199.33$ \\
		\hline
	\end{tabular}
	\caption{Phase velocity $c_s$ and the corresponding time period $2\pi/c_s$ for different Lam\'{e} parameters and densities.}
	\label{tab_Stoneley_parameter}
\end{table}

\begin{figure}[htbp]
	\centering
	\includegraphics[width=0.48\textwidth,trim={0.5cm 0cm 1cm 0cm}, clip]{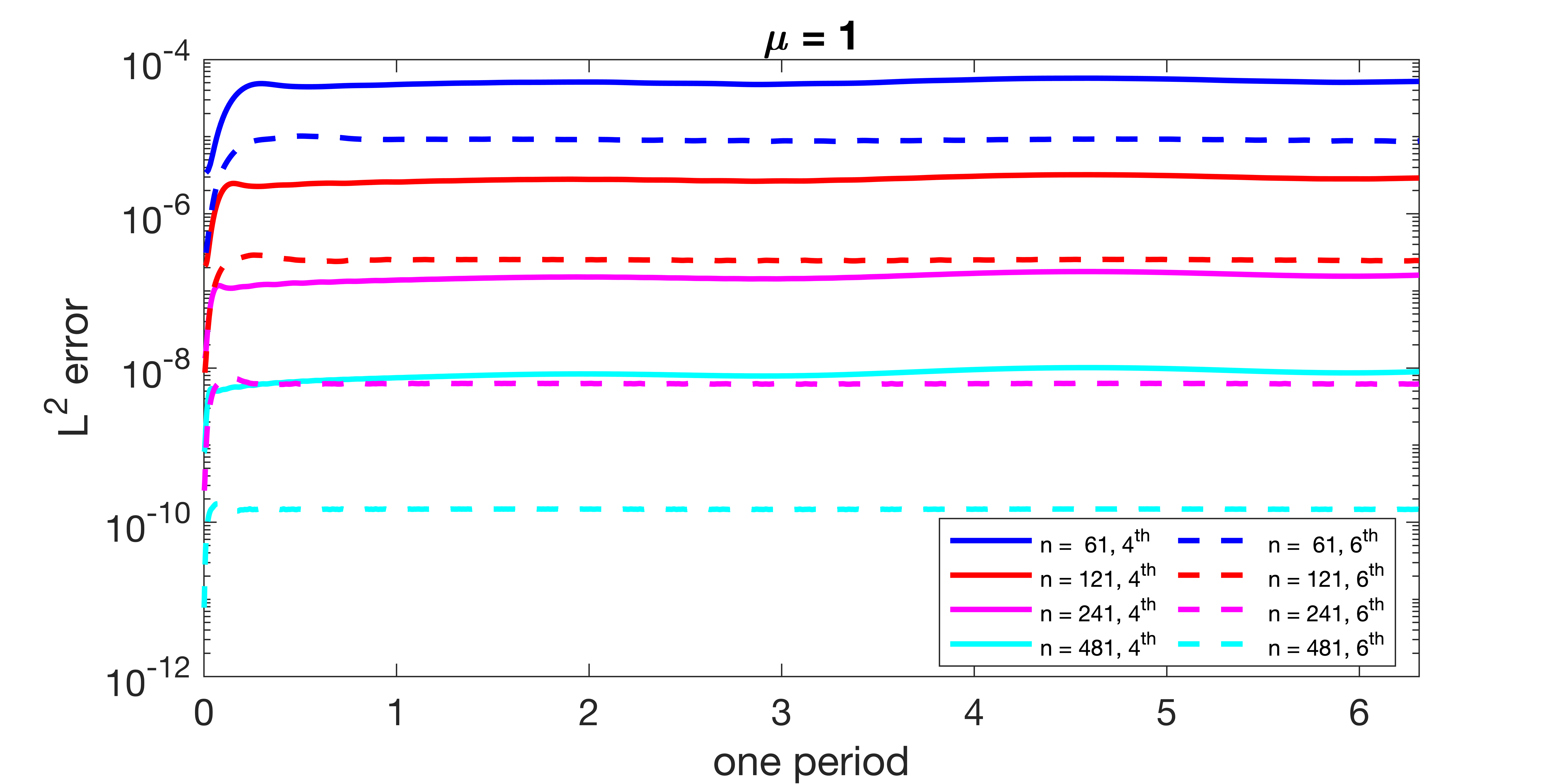}
	\includegraphics[width=0.48\textwidth,trim={0.5cm 0cm 1cm 0cm}, clip]{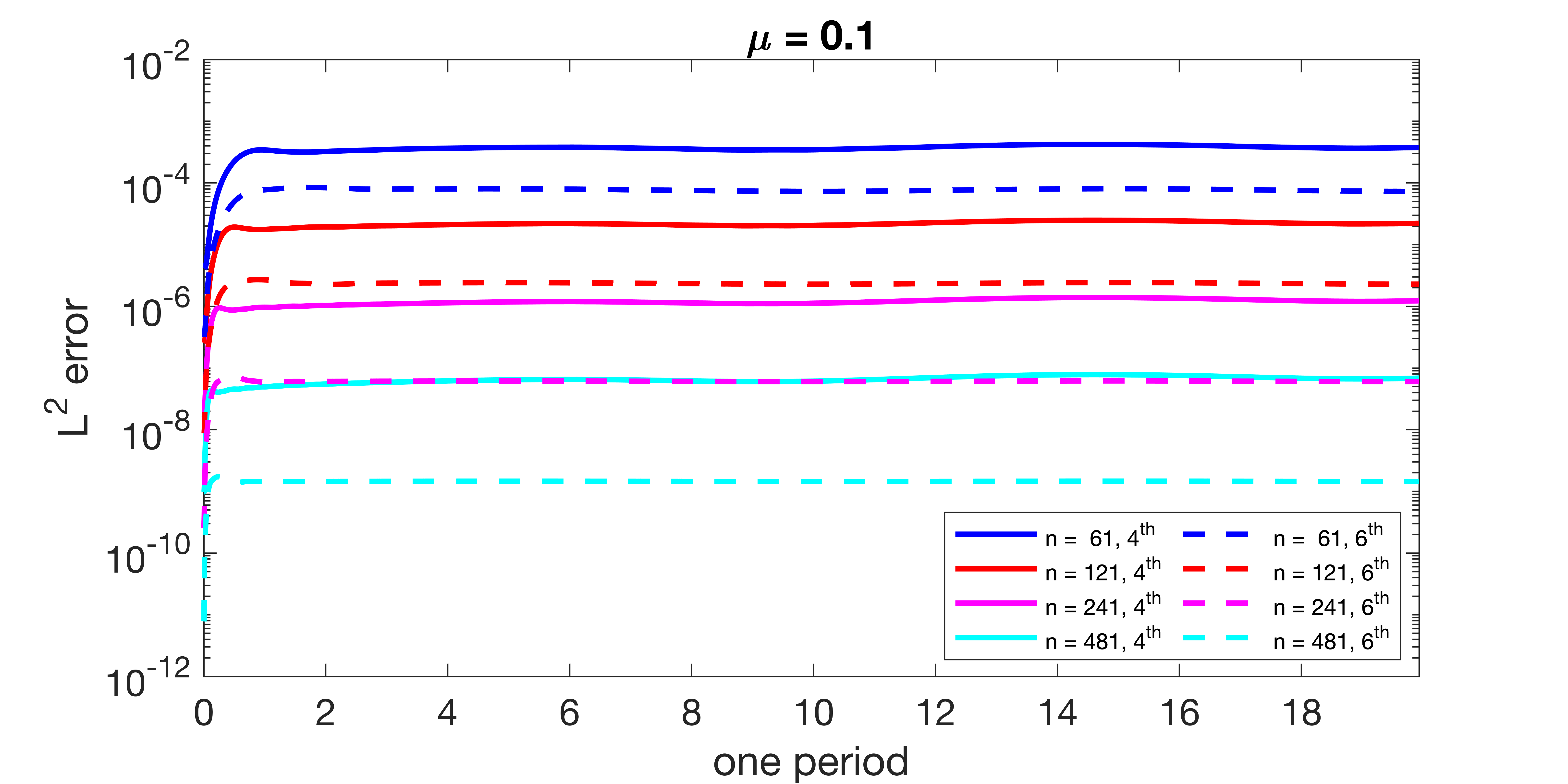}\\
	\includegraphics[width=0.48\textwidth,trim={0.5cm 0cm 1cm 0cm}, clip]{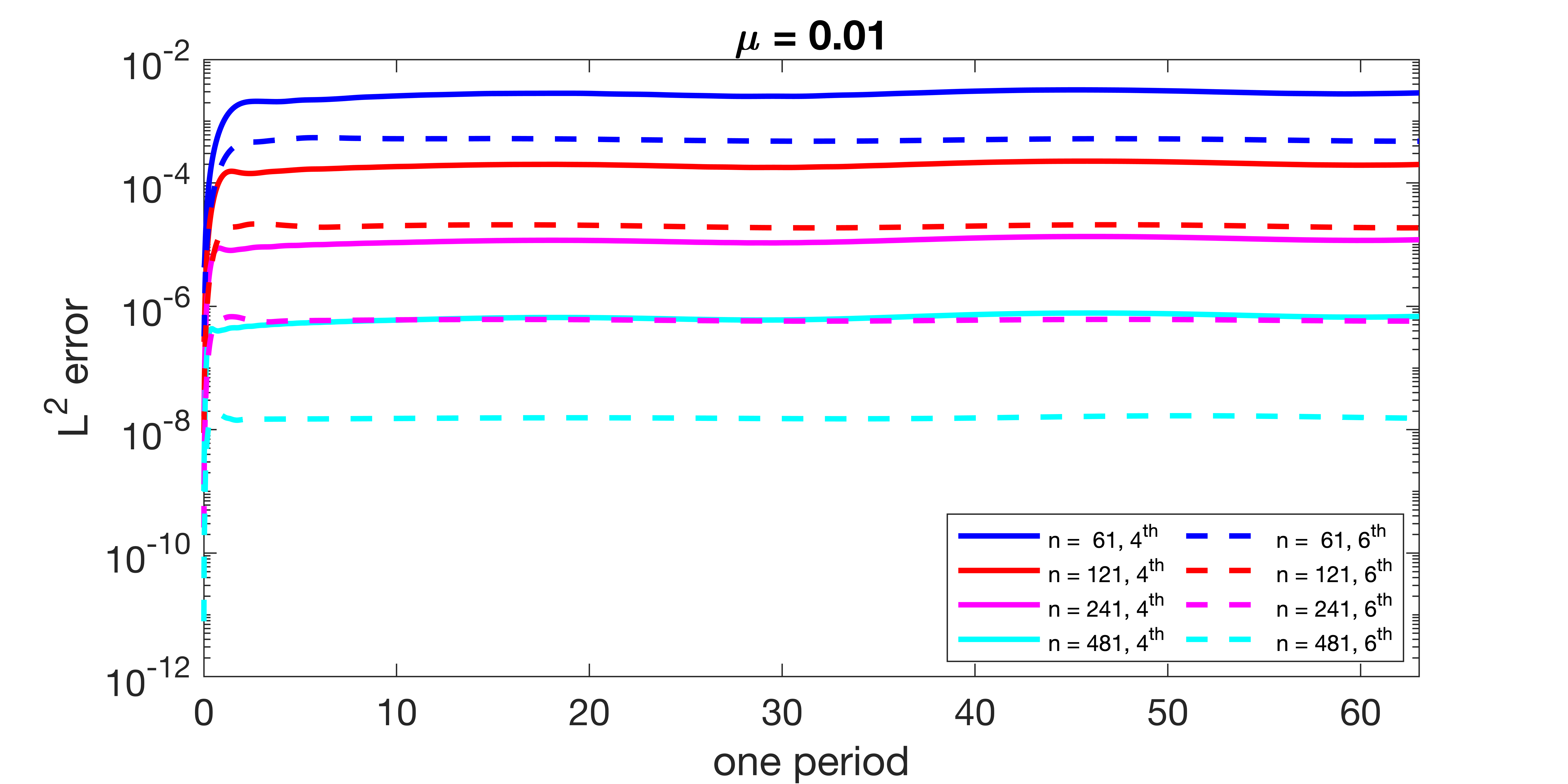}
	\includegraphics[width=0.48\textwidth,trim={0.5cm 0cm 1cm 0cm}, clip]{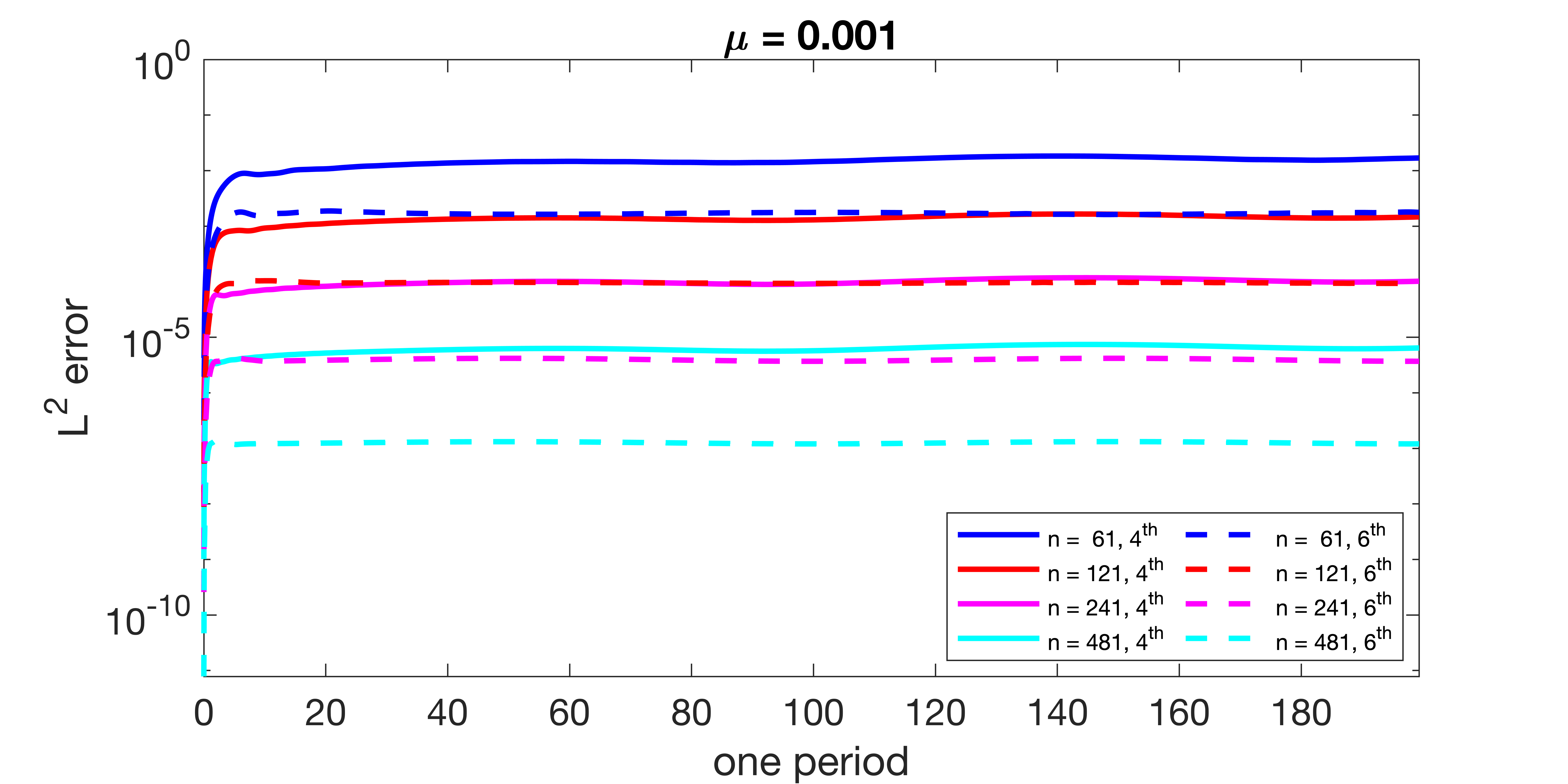}
	\caption{$l^2$ error history for the Stoneley wave problem with both fourth and sixth order schemes. From the left to the right, the top to the bottom are for $\mu = 1, 0.1, 0.01$ and $0.001$, respectively. The solid lines are for the fourth order method, and the dash lines are for the sixth order method.
	}\label{error_history}
\end{figure}

\begin{table}
	\centering
	\begin{tabular}{lllll}
		\hline
		$n$ & $l^2$ error (4th) & convergence rate & $l^2$ error (6th) & convergence rate \\
		\hline
		61 & $5.2044\times 10^{-5}$ & &$8.7899\times 10^{-6}$ & \\
		121 & $2.9083\times 10^{-6}$ & 4.1614 &$2.4716\times 10^{-7}$ & 5.1523 \\
		241 & $1.5990\times 10^{-7}$ & 4.1849& $6.2267\times 10^{-9}$& 5.3108\\
		481 &  $8.9610\times 10^{-9}$& 4.1573& $1.4671\times 10^{-10}$&  5.4074\\
		\hline
	\end{tabular}
	\caption{$l^2$ errors and convergence rates for the Stoneley wave problem with $\mu = 1$ after one temporal period.}
	\label{tab_Stoneley_mu1}
\end{table}

\begin{table}
	\centering
	\begin{tabular}{lllll}
		\hline
		$n$ & $l^2$ error (4th) & convergence rate & $l^2$ error (6th) & convergence rate \\
		\hline
		61 & $3.7483\times 10^{-4}$ &  & $7.2833\times 10^{-5}$& \\
		121 & $2.2121\times 10^{-5}$ & 4.0826& $2.2862\times 10^{-6}$&  4.9935\\
		241 & $1.2294\times 10^{-6}$ &4.1693 &  $6.0084\times 10^{-8}$& 5.2498\\
		481 &  $6.8408\times 10^{-8}$& 4.1676& $1.4437\times 10^{-9}$&  5.3791\\
		\hline
	\end{tabular}
	\caption{$l^2$ errors and convergence rates for the Stoneley wave problem with $\mu = 0.1$ after one temporal period.}
	\label{tab_Stoneley_mu01}
\end{table}

\begin{table}
	\centering
	\begin{tabular}{lllll}
		\hline
		$n$ & $L_2$ error (4th) & convergence rate & $L_2$ error (6th) & convergence rate \\
		\hline
		61 & $2.8806\times 10^{-3}$ &  &$4.7402\times 10^{-4}$ & \\
		121 & $1.9891\times 10^{-4}$ &3.8561 &$1.8704\times 10^{-5}$ & 4.6635 \\
		241 & $1.1996\times 10^{-5}$ & 4.0514& $5.7506\times10^{-7}$ & 5.0234\\
		481 &  $6.8723\times 10^{-7}$& 4.1256& $1.5336\times 10^{-8}$&  5.2287\\
		\hline
	\end{tabular}
	\caption{$L_2$ errors and convergence rates for the Stoneley wave problem with $\mu = 0.01$ after one temporal period.}
	\label{tab_Stoneley_mu001}
\end{table}

The main purpose of the numerical experiments in this section is to study the convergence properties of the proposed fourth and sixth order schemes when simulating Stoneley waves \revi{with discontinuous material property}. The computational domain is taken to contain exactly one wavelength, $2\pi$, in the $x$-direction. Exact Dirichlet data are imposed at physical boundaries. Initial data for the computations is taken from (\ref{stoneley_f}) and (\ref{stoneley_c}) at time $t = 0$. We choose $n_1^{2h}=n_2^{2h} = n$, $n_1^{h} = n_2^{h}= 2n-1$ for all simulations. The cumulative $L_2$ errors are computed as a function of time for one temporal period. In Table \ref{tab_Stoneley_parameter}, we present four different sets of Lam\'{e} parameters and densities, and their corresponding phase velocity $c_s$ and time period. As $\mu$ becomes smaller and smaller and $\lambda^f \gg \mu^f=\mu$, $\lambda^c \gg \mu^c = 2\mu$, the Stoneley interface wave is similar to the Rayleigh wave. It has been shown by \cite{duru2014accurate,kreiss2012boundary} that numerical simulations for the Stoneley wave which is similar to the Rayleigh wave require an unexpectedly large resolution for accurate solutions. The experiments in this section also evaluate the performance of the proposed fourth and sixth methods when simulating Stoneley waves in this extreme case.

Figure \ref{error_history} presents the cumulative $L_2$ error history for one temporal period obtained by both fourth and sixth order methods with the material parameters given in Table \ref{tab_Stoneley_parameter}. The corresponding convergence rate for each set of material parameter values are given from Table \ref{tab_Stoneley_mu1} to Table \ref{tab_Stoneley_mu0001}. We observe that the cumulative $L_2$ errors for the sixth order method are smaller than the errors for the fourth order method for all cases, but both high order methods resolve the solution well. Further, for both fourth order and sixth \revii{order} methods, we observe a slightly better convergence rate comparing with the theoretical rate: four for the fourth order method and five for the sixth order method. 

\begin{table}
\centering
\begin{tabular}{lllll}
\hline
$n$ & $L_2$ error (4th) & convergence rate & $L_2$ error (6th) & convergence rate \\
\hline
61 & $1.6901\times 10^{-2}$ &  & $1.7712\times 10^{-3}$& \\
121 & $1.4710\times 10^{-3}$ &3.5222 & $9.3318\times 10^{-5}$  & 4.2464 \\
241 & $1.0204\times 10^{-4}$ &3.8495 & $3.6962\times 10^{-6}$ & 4.6580 \\
481 & $6.4232\times 10^{-6}$ &3.9896& $1.2041\times 10^{-7}$& 4.9400 \\
\hline
\end{tabular}
\caption{$L_2$ errors and convergence rates for the Stoneley wave problem with $\mu=0.001$ after one temporal period.}
\label{tab_Stoneley_mu0001}
\end{table}

\subsection{Energy conservation}\label{sec:energy_conservation}
To verify the energy conservation property of the proposed fourth and sixth order schemes, we perform computation without external source term, but with a Gaussian initial data centered at the origin of the computational domain. The computational domain is chosen to be the same as in Sec.~\ref{sec_manufactured_sol}. The material property is heterogeneous and discontinuous: for the fine domain $\Omega^f$, the density varies according to
\[\rho^f(x, y) = 4 + \sin(x + 0.3)\sin(y - 0.2),\]
and material parameters satisfy
\[\mu^f(x, y) = 3 + \sin(3x + 0.1)\sin(y),\ \ \ \lambda^f(x, y) = 15 + \cos(x + 0.1)\sin^2(3y);\]
for the coarse domain $\Omega^c$, the density varies according to
\[\rho^c(x, y) = 2 + \sin(4x + 0.3)\sin(y - 0.2),\]
and material parameters satisfy
\[\mu^c(x, y) = 3 + \sin(3x + 0.1)\sin(2y),\ \ \ \lambda^c(x, y) = 21 + \cos(x + 0.1)\sin^2(3y).\]
\revii{For the initial condition, we sample the data on every grid point from the uniform distribution between 0 and 1. Before the time loop, we modify the initial data on the interface and on the ghost point such that the interface conditions \eqref{continuous_sol} and \eqref{continuous_traction} are satisfied. }

In Figure \ref{energy_conserved}, we present the relative change in the fully discrete energy (\ref{discrete_energy}), $(E^{n+1/2} - E^{1/2})/E^{1/2}$, as a function of time with
$t \in [0, 100]$, $n_1^{2h} = 61$, $n_2^{2h} = 31$, $n_1^h = 121$ and $n_1^{2h} = 61$ for both fourth and sixth schemes. Clearly, we observe that the fully discrete energy remains constant up to the round-off error, even with the highly oscillatory initial data.

\begin{figure}[htbp]
	\centering
	\includegraphics[width=0.48\textwidth,trim={0.2cm 0.2cm 1cm 0.1cm}, clip]{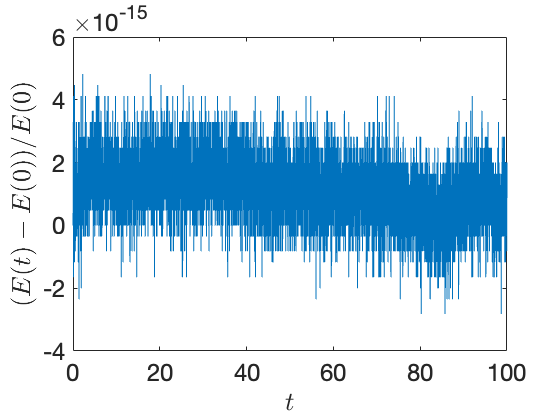}
	\includegraphics[width=0.48\textwidth,trim={0.1cm 0.2cm 1cm 0.1cm}, clip]{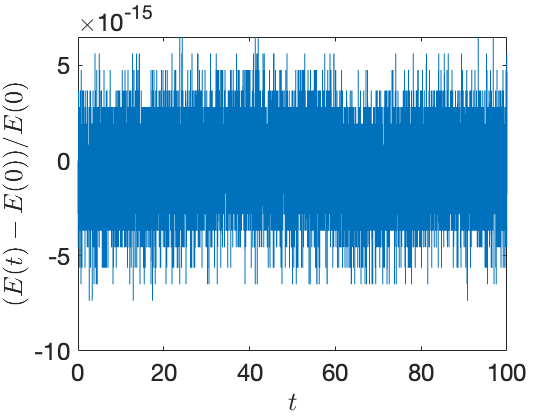}
	\caption{The relative change in the fully discrete energy as a function of time. The left panel is for the fourth order method and the right panel is for the sixth order method.
	}\label{energy_conserved}
\end{figure}

\subsection{Comparison with the SAT method}\label{cfl_sat}
In the last numerical example, we compare the proposed methods with the SBP-SAT method in terms of CFL time-step restriction. Since the SBP-SAT discretization for the elastic wave equations with nonconforming interfaces is not available in the literature, we use the scalar wave equation as the model problem. More precisely, we consider the example in Sec.~7.3 from \cite{Almquist2019} with piecewise constant material property, 
\[
u_{tt} = C^2 (u_{xx}+ u_{yy}),
\]
where the wave speed is $C=1$ in $(x,y)\in \Omega_c:=[-10,0]\times[0,10]$ and  $C=0.5$ in $(x,y)\in \Omega_f:=[0,10]\times[0,10]$. At the interface $x=0$, $y\in [0,10]$, we impose interface conditions as continuity of $u$ and $C^2u_x$. In this case, an analytical solution can be constructed by Snell's law, which takes the form  
\begin{equation}\label{snell}
u = \begin{cases}
\cos(x+y-\sqrt{2}t)+k_2\cos(x-y+\sqrt{2}t), \quad (x,y)\in\Omega_c \\
(1+k_2)\cos(k_1x+y-\sqrt{2}t),\quad (x,y)\in\Omega_f,
\end{cases}
\end{equation}
where $k_1=\sqrt{7}$ and $k_2=(4-\sqrt{7})/(4+\sqrt{7})$. In the convergence test, we use the analytical solution \eqref{snell} to obtain the data for initial conditions and the Dirichlet boundary conditions at all boundaries. 

In the spatial discretization, we use $121^2$ grid points in the coarse domain $\Omega_c$ and $241^2$ grid points in the fine domain $\Omega_f$. In this case, the number of grid points per wavelength is the same in both $\Omega_c$ and $\Omega_f$. We denote the restriction on time-step $\Delta_t$ as
\begin{equation}\label{cfl}
\Delta_t\leq \theta C_{vn}\frac{h}{\sqrt{2}c},
\end{equation}
where $C_{vn}$ is derived by the Von Neumann analysis for the corresponding periodic problem, and $\theta\leq 1$ is computed numerically as the largest value such that the numerical method is stable. \revii{More precisely, we start with an initial value of 1.05 for $\theta$, and solve the equation until the final time $T=20$. If it is unstable, we decrease $\theta$ by 0.01 and repeat this procedure until the method is stable, and report the final $\theta$ values in Table \ref{tab_theta}.} With the fourth order predictor-corrector time stepping method in Sec.~\ref{time_integrator}, it is shown in \cite{wang2018fourth} that $C_{vn}=1.5$ for the fourth order spatial discretization. By the same approach, $C_{vn}=2\sqrt{135/272}$ for the sixth order spatial discretization for periodic problems. In the SBP-SAT method, the penalty parameters are chosen to be three times the limit value required by stability for both Dirichlet boundaries and the grid interface, \revi{which is the same choice as in \cite{Almquist2019} when solving the wave equation}.

\begin{table}
	\centering
	\begin{tabular}{ccc}
		\hline
		& SBP-GP & SBP-SAT\\
		\hline
		fourth order  & 0.98& 0.18 \\
		sixth order  & 0.65& 0.26 \\
		\hline
	\end{tabular}
	\caption{The value $\theta$ in \eqref{cfl}. }
	\label{tab_theta}
\end{table}

For the fourth order method, the SBP-GP discretization allows for a time step nearly as large as given by the Von Neumann analysis, but the time step for the SBP-SAT discretization must be lowered five times. 
For the sixth order method,  we have $\theta=0.65$ for the SBP-GP discretization. This indicates that we have to use a smaller time step than for periodic problems. However, this is not due to the interface coupling. By numerically computing the spectral radius of the sixth order SBP operator, we find that it is more than twice large the spectral radius for the corresponding central difference operator for periodic problems. In the SBP-SAT discretization, the time step needs to reduce further for stability.

\section{Conclusion}
We have developed a high order multiblock finite difference method for the elastic wave equations on curvilinear grids. The finite difference stencils satisfy a summation-by-parts property. Adjacent blocks are coupled by imposing interface conditions strongly using ghost points. Grid sizes can be chosen differently across material interfaces to adapt to the velocity structure of the material. In our previous work \cite{Zhang2021} with periodic or far-field boundary conditions in the spatial direction tangential to the material interface, interpolation operators based on centered stencils are used for the hanging nodes. However, for problems in bounded domains, additional difficulties arise in both stability and accuracy, because interpolation stencils must be modified near the edges of the grid interface. In this case, we use the order preserving interpolation operators from \cite{Almquist2019} for the hanging nodes. Our truncation error analysis shows that only one pair of interpolation operators is needed, instead of two pairs when interface conditions are imposed weakly using penalty terms for the wave equation. Together with an explicit predictor-corrector time stepping method, we have proved that the fully discrete energy is conserved in time. The CFL time step restriction is not affected by the numerical interface coupling. 

We have presented numerical examples to verify the energy conserving property and the excellent time-step restriction of the proposed method.  In addition, by the manufactured solution approach with smooth material property and a curved interface, we have observed that the convergence rate is four and five, when the interior order of the finite difference stencil is four and six, respectively. The same convergence rates are also observed in the Stoneley wave example, where the interface geometry is flat but the material property is discontinuous. In both cases, different grid sizes are used across a nonconforming grid interface, and they do not lead to any accuracy reduction.

\section*{Acknowledgement}
Part of the computations were performed on resources provided by Swedish National Infrastructure for Computing (SNIC) through Uppsala Multidisciplinary Center for Advanced Computational Science
 (UPPMAX) under Project SNIC 2019/8-263.

\appendix

\section{Positivity of the fully discrete energy $E^{n+1/2}$}\label{positive_energy}
Let $\tau$ be the summation of the third and fourth terms on the right-hand side of (\ref{discrete_energy}). By  the symmetry of $\mathcal{S}_h$ and $\mathcal{S}_{2h}$, we have
\begin{align*}
\tau &:=\mathcal{S}_{2h}({\bf c}^{n+1}, {\bf c}^{n})+\mathcal{S}_h({\bf f}^{n+1}, {\bf f}^n) \\
&= \frac{1}{2}\left[\mathcal{S}_h({\bf f}^{n+1}, {\bf f}^n) + \mathcal{S}_{2h}({\bf c}^{n+1}, {\bf c}^{n})\right]  + \frac{1}{2}\left[\mathcal{S}_h({\bf f}^{n}, {\bf f}^{n+1}) + \mathcal{S}_{2h}({\bf c}^{n}, {\bf c}^{n+1})\right].
\end{align*}
We then use the relations (\ref{sbp_c}) and (\ref{sbp_f}), the fact that both $\{{\bf f}^{n+1}, {\bf c}^{n+1}\}$ and $\{{\bf f}^n, {\bf c}^n\}$ satisfy the interface conditions (\ref{continuous_sol}) and (\ref{continuous_traction}), \revii{ and Remark \ref{remark1}} to obtain
\begin{multline}\label{eq1}
\tau = -\frac{1}{2} \left[\big({\bf f}^{n+1}, (\mathcal{J}^h)^{-1}\hat{\mathcal{L}}^h{\bf f}^n\big)_{hw} + \big({\bf c}^{n+1}, (\mathcal{J}^{2h})^{-1}\wt{\mathcal{L}}^{2h}{\bf c}^n\big)_{2hw} \right] \\
 - \frac{1}{2} \left[\big({\bf f}^{n}, (\mathcal{J}^h)^{-1}\hat{\mathcal{L}}^h{\bf f}^{n+1}\big)_{hw} + \big({\bf c}^{n}, (\mathcal{J}^{2h})^{-1}\wt{\mathcal{L}}^{2h}{\bf c}^{n+1}\big)_{2hw} \right].
\end{multline}
Note that the above derivation is analog to the analysis leading $\sigma_1 = 0$ in Theorem \ref{energy_conserve_th}, and will be repeated later in the proof as well. 

Next, we denote the last two terms in the right-hand side of (\ref{discrete_energy})   by $\gamma$, and rewrite as 
\begin{multline*}
\gamma = -\frac{\Delta_t^2}{24}\left[\left((\mathcal{J}^h_\rho)^{-1}\hat{\mathcal{L}}^h{\bf f}^{n+1}, (\mathcal{J}^h)^{-1}\hat{\mathcal{L}}^h{\bf f}^{n}\right)_{hw} + \left((\mathcal{J}_{\rho}^{2h})^{-1}\wt{\mathcal{L}}^{2h}{\bf c}^{n+1}, (\mathcal{J}^{2h})^{-1}\wt{\mathcal{L}}^{2h}{\bf c}^{n}\right)_{2hw}\right] \\
-\frac{\Delta_t^2}{24}\left[\left((\mathcal{J}_\rho^h)^{-1}\hat{\mathcal{L}}^h{\bf f}^{n}, (\mathcal{J}^h)^{-1}\hat{\mathcal{L}}^h{\bf f}^{n+1}\right)_{hw} + \left((\mathcal{J}_\rho^{2h})^{-1}\wt{\mathcal{L}}^{2h}{\bf c}^{n}, (\mathcal{J}^{2h})^{-1}\wt{\mathcal{L}}^{2h}{\bf c}^{n+1}\right)_{2hw}\right].
\end{multline*}
In the following two steps, we use the relations (\ref{sbp_c}) and (\ref{sbp_f}) twice to further rewrite $\gamma$ and move the spatial difference operators to one side in the inner products. 

Step 1: for the first two terms in $\gamma$, we replace ${\bf v}_h$ by $(\mathcal{J}_\rho^h)^{-1}\hat{\mathcal{L}}^h{\bf f}^{n+1}$ and ${\bf u}_h$ by ${\bf f}^n$ in (\ref{sbp_f}), ${\bf v}_{2h}$ by $(\mathcal{J}_\rho^{2h})^{-1}\wt{\mathcal{L}}^{2h}{\bf c}^{n+1}$ and ${\bf u}_{2h}$ by ${\bf c}^n$ in (\ref{sbp_c}); for the last two terms in $\gamma$, we replace ${\bf v}_h$ by $(\mathcal{J}_\rho^h)^{-1}\hat{\mathcal{L}}^h{\bf f}^{n}$ and ${\bf u}_h$ by ${\bf f}^{n+1}$ in (\ref{sbp_f}), ${\bf v}_{2h}$ by $(\mathcal{J}_\rho^{2h})^{-1}\wt{\mathcal{L}}^{2h}{\bf c}^{n}$ and ${\bf u}_{2h}$ by ${\bf c}^{n+1}$ in (\ref{sbp_c}). \revii{Recalling (\ref{predictor_stable}), we have
	\[ (\mathcal{J}^{2h}_\rho)^{-1} \widetilde{\mathcal{L}}^{2h}{\bf c}^{n+1} = \frac{{\bf c}^{\ast, n+2} - 2{\bf c}^{n+1} + {\bf c}^{n}}{\Delta t^2}, \quad (\mathcal{J}^h_\rho)^{-1} \hat{\mathcal{L}}^h{\bf f}^{n+1} = \frac{{\bf f}^{\ast, n+2} - 2{\bf f}^{n+1} + {\bf f}^{n}}{\Delta t^2}, \]
	and
		\[ (\mathcal{J}^{2h}_\rho)^{-1} \widetilde{\mathcal{L}}^{2h}{\bf c}^{n} = \frac{{\bf c}^{\ast, n+1} - 2{\bf c}^{n} + {\bf c}^{n-1}}{\Delta t^2}, \quad (\mathcal{J}^h_\rho)^{-1} \hat{\mathcal{L}}^h{\bf f}^{n} = \frac{{\bf f}^{\ast, n+1} - 2{\bf f}^{n} + {\bf f}^{n-1}}{\Delta t^2}.\]
	}Requiring that $\{{\bf f}^{\ast,n+2}, {\bf c}^{\ast,n+2}\}$, $\{{\bf f}^{\ast,n+1}, {\bf c}^{\ast,n+1}\}$, $\{{\bf f}^{n+1}, {\bf c}^{n+1}\}$, $\{{\bf f}^n, {\bf c}^n\}$, $\{{\bf f}^{n-1}, {\bf c}^{n-1}\}$ satisfy the interface condition (\ref{continuous_sol}), and $\{{\bf f}^n, {\bf c}^n\}$, $\{{\bf f}^{n+1}, {\bf c}^{n+1}\}$ satisfy the interface condition (\ref{continuous_traction}), we obtain 
\begin{multline*}
\gamma = \frac{\Delta_t^2}{24}\left(\mathcal{S}_h\big((\mathcal{J}_\rho^h)^{-1}\hat{\mathcal{L}}^h{\bf f}^{n+1}, {\bf f}^{n}\big) + \mathcal{S}_{2h}\big((\mathcal{J}_\rho^{2h})^{-1}\wt{\mathcal{L}}^{2h}{\bf c}^{n+1}, {\bf c}^{n}\big)\right) \\
+ \frac{\Delta_t^2}{24}\left(\mathcal{S}_h\big((\mathcal{J}_\rho^h)^{-1}\hat{\mathcal{L}}^h{\bf f}^{n}, {\bf f}^{n+1}\big) + \mathcal{S}_{2h}\big((\mathcal{J}_\rho^{2h})^{-1}\wt{\mathcal{L}}^{2h}{\bf c}^{n}, {\bf c}^{n+1}\big)\right).
\end{multline*}
Note that the above derivation is again similar to the analysis in Theorem \ref{energy_conserve_th} for $\sigma_1 = 0$.

Step 2: by using the symmetric property of $\mathcal{S}_h$, $\mathcal{S}_{2h}$, and then using the relations (\ref{sbp_c}), (\ref{sbp_f}) one more time: take ${\bf v}_h$ to be ${\bf f}^n$,  ${\bf u}_h$ to be $(\mathcal{J}_\rho^h)^{-1}\hat{\mathcal{L}}^h{\bf f}^{n+1}$ in (\ref{sbp_f}), ${\bf v}_{2h}$ to be ${\bf c}^n$,  ${\bf u}_{2h}$ to be $(\mathcal{J}_\rho^{2h})^{-1}\wt{\mathcal{L}}^{2h}{\bf c}^{n+1}$ in (\ref{sbp_c}) for the first two terms in $\gamma$; take ${\bf v}_h$ to be ${\bf f}^{n+1}$,  ${\bf u}_h$ to be $(\mathcal{J}_\rho^h)^{-1}\hat{\mathcal{L}}^h{\bf f}^{n}$ in (\ref{sbp_f}), ${\bf v}_{2h}$ to be ${\bf c}^{n+1}$,  ${\bf u}_{2h}$ to be $(\mathcal{J}_\rho^{2h})^{-1}\wt{\mathcal{L}}^{2h}{\bf c}^{n}$ in (\ref{sbp_c}) for the last two terms in $\gamma$. \revii{ Recalling (\ref{predictor_stable}) and }requiring that $\{{\bf f}^n, {\bf c}^n\}$, $\{{\bf f}^{n+1}, {\bf c}^{n+1}\}$ satisfy the interface condition (\ref{continuous_sol}), 
$\{{\bf f}^{\ast,n+2}, {\bf c}^{\ast,n+2}\}$,  $\{{\bf f}^{\ast,n+1}, {\bf c}^{\ast,n+1}\}$, $\{{\bf f}^{n+1}, {\bf c}^{n+1}\}$, $\{{\bf f}^n, {\bf c}^n\}$ and $\{{\bf f}^{n-1},{\bf c}^{n-1}\}$ satisfy the interface condition (\ref{continuous_traction}),  we have 
\begin{multline}\label{eq2}
\gamma = -\frac{\Delta_t^2}{24}\left[\left({\bf f}^{n}, (\mathcal{J}^h)^{-1}\hat{\mathcal{L}}^h\big((\mathcal{J}_\rho^h)^{-1}\hat{\mathcal{L}}^h{\bf f}^{n+1}\big)\right)_{hw} + \left( {\bf c}^{n}, (\mathcal{J}^{2h})^{-1}\wt{\mathcal{L}}^{2h}\big((\mathcal{J}_\rho^{2h})^{-1}\wt{\mathcal{L}}^{2h}{\bf c}^{n+1}\big)\right)_{2hw}\right] \\
-\frac{\Delta_t^2}{24}\left[\left({\bf f}^{n+1}, (\mathcal{J}^h)^{-1}\hat{\mathcal{L}}^h\big((\mathcal{J}_\rho^h)^{-1}\hat{\mathcal{L}}^h{\bf f}^{n}\big)\right)_{hw} + \left( {\bf c}^{n+1}, (\mathcal{J}^{2h})^{-1}\wt{\mathcal{L}}^{2h}\big((\mathcal{J}_\rho^{2h})^{-1}\wt{\mathcal{L}}^{2h}{\bf c}^{n}\big)\right)_{2hw}\right].
\end{multline}
We have now moved the spatial difference operators to one side in the inner products in $\gamma$. Substituting (\ref{eq1}) and (\ref{eq2}) into (\ref{discrete_energy}), we obtain
	\begin{multline}\label{eq3}
	E^{n+1/2} 
	= \left({\bf f}^{n+1} - {\bf f}^n, \left[\frac{1}{\Delta_t^2}\varrho^h + \frac{1}{4}(\mathcal{J}^h)^{-1}\hat{\mathcal{L}}^h + \frac{\Delta_t^2}{48}(\mathcal{J}^h)^{-1}\hat{\mathcal{L}}^h\left((\mathcal{J}_\rho^h)^{-1}\hat{\mathcal{L}}^h\right)\right]\left({\bf f}^{n+1} - {\bf f}^n\right)\right)_{hw}\\
	+\left({\bf c}^{n+1} - {\bf c}^n, \left[\frac{1}{\Delta_t^2}\varrho^{2h}+ \frac{1}{4}(\mathcal{J}^{2h})^{-1}\wt{\mathcal{L}}^{2h} + \frac{\Delta_t^2}{48}(\mathcal{J}^{2h})^{-1}\wt{\mathcal{L}}^{2h}\left((\mathcal{J}_\rho^{2h})^{-1}\wt{\mathcal{L}}^{2h}\right)\right]\left({\bf c}^{n+1} - {\bf c}^n\right)\right)_{2hw}\\
	+ \left({\bf f}^{n+1} + {\bf f}^n, \left[-\frac{1}{4}(\mathcal{J}^h)^{-1}\hat{\mathcal{L}}^h - \frac{\Delta_t^2}{48}(\mathcal{J}^{h})^{-1}\hat{\mathcal{L}}^h\left((\mathcal{J}_\rho^h)^{-1}\hat{\mathcal{L}}^h\right)\right]\left({\bf f}^{n+1} + {\bf f}^n\right)\right)_{hw}\\
	+ \left({\bf c}^{n+1} + {\bf c}^n, \left[-\frac{1}{4}(\mathcal{J}^{2h})^{-1}\wt{\mathcal{L}}^{2h} - \frac{\Delta_t^2}{48}(\mathcal{J}^{2h})^{-1}\wt{\mathcal{L}}^{2h}\left((\mathcal{J}_\rho^{2h})^{-1}\wt{\mathcal{L}}^{2h}\right)\right]\left({\bf c}^{n+1} + {\bf c}^n\right)\right)_{2hw}.
	\end{multline}

Now we rewrite the spatial difference operators in \eqref{eq3} by using the relation \eqref{K}. We note that for any grid functions $\{{\bf u}_{2h} , {\bf u}_h\}$ and $\{{\bf v}_{2h}, {\bf v}_h\}$ satisfying the interface conditions, we have the identity


	\[\left({\bf v}_h, (\mathcal{J}^h)^{-1}\hat{\mathcal{L}}^h{\bf u}_{h}\right)_{hw} + \left({\bf v}_{2h}, (\mathcal{J}^{2h})^{-1}\wt{\mathcal{L}}^{2h}{\bf u}_{2h}\right)_{2hw} 
	= -\mathcal{S}_h({\bf v}_h, {\bf u}_h) -\mathcal{S}_{2h}({\bf v}_{2h}, {\bf u}_{2h}).\]
	Using the relation (\ref{K}), we can easily obtain
	\begin{equation}\label{eq6a}
	\left({\bf v}_h, (\mathcal{J}^h)^{-1}\hat{\mathcal{L}}^h{\bf u}_{h}\right)_{hw} + \left({\bf v}_{2h}, (\mathcal{J}^{2h})^{-1}\wt{\mathcal{L}}^{2h}{\bf u}_{2h}\right)_{2hw} 
	= - \left({\bf u}_h, \varrho^hK_h{\bf v}_h\right)_{hw} - \left({\bf u}_{2h}, \varrho^{2h}K_{2h}{\bf v}_{2h}\right)_{2hw}.
	\end{equation}
	Thus, (\ref{eq3}) can be written as
	\begin{multline}\label{eq6}
	E^{n+1/2} 
	\!=\!\left(\!\delta {\bf f}^n, \left(\!\frac{\varrho^h}{\Delta_t^2} - \frac{\varrho^h}{4}K_h + \frac{\Delta_t^2}{48}\varrho^hK_h^2\!\right)\delta {\bf f}^n\!\right)_{hw}
	\!\!\!+\left(\!\delta {\bf c}^n, \left(\!\frac{\varrho^{2h}}{\Delta_t^2} - \frac{\varrho^{2h}}{4}K_{2h} + \frac{\Delta_t^2}{48}\varrho^{2h}K_{2h}^2\!\right)\delta {\bf c}^n\!\right)_{2hw}\\
	+ \left(\!{\bf f}^{n+1} \!+ {\bf f}^n, \left(\!\frac{\varrho^h}{4}K_h \!-\! \frac{\Delta_t^2}{48}\varrho^hK_h^2\!\right)\!\left({\bf f}^{n+1} \!+ {\bf f}^n\right)\!\right)_{hw}
	\!\!\!+ \left(\!{\bf c}^{n+1} \!+ {\bf c}^n, \left(\!\frac{\varrho^{2h}}{4}K_{2h} \!-\! \frac{\Delta_t^2}{48}\varrho^{2h}K_{2h}^2\!\right)\!\left({\bf c}^{n+1} \!+ {\bf c}^n\right)\!\right)_{2hw},
	\end{multline}
	where we have used 
	\begin{align*}
	&\hspace{0.2cm}\left({\bf f}^{n+1} \!\!-\! {\bf f}^n, (\mathcal{J}^h)^{-1}\hat{\mathcal{L}}^h\big((\mathcal{J}_\rho^h)^{-1}\hat{\mathcal{L}}^h({\bf f}^{n+1} \!\!-\! {\bf f}^n)\big)\right)_{hw} \!\!+ \left({\bf c}^{n+1} \!\!-\! {\bf c}^n, (\mathcal{J}^{2h})^{-1}\wt{\mathcal{L}}^{2h}\big((\mathcal{J}_\rho^{2h})^{-1}\wt{\mathcal{L}}^{2h}({\bf c}^{n+1} \!\!-\! {\bf c}^n)\big)\right)_{2hw} \\
	&= -\left((\mathcal{J}_\rho^h)^{-1}\hat{\mathcal{L}}^h({\bf f}^{n+1} - {\bf f}^n), \varrho^hK_h({\bf f}^{n+1} - {\bf f}^n)\right)_{hw} - \left((\mathcal{J}_\rho^{2h})^{-1}\wt{\mathcal{L}}^{2h}({\bf c}^{n+1} - {\bf c}^n), \varrho^{2h}K_{2h}({\bf c}^{n+1} - {\bf c}^n)\right)_{2hw}\\
	&= -\left(K_h({\bf f}^{n+1} - {\bf f}^n),(\mathcal{J}^h)^{-1}\hat{\mathcal{L}}^h({\bf f}^{n+1} - {\bf f}^n)\right)_{hw} - \left(K_{2h}({\bf c}^{n+1} - {\bf c}^n),(\mathcal{J}^{2h})^{-1}\wt{\mathcal{L}}^{2h}({\bf c}^{n+1} - {\bf c}^n)\right)_{2hw}\\
	&= \left({\bf f}^{n+1} - {\bf f}^n, \varrho^hK^2_h({\bf f}^{n+1} - {\bf f}^n)\right)_{hw} + \left({\bf c}^{n+1} - {\bf c}^n, \varrho^{2h}K^2_{2h}({\bf c}^{n+1} - {\bf c}^n)\right)_{2hw}.
	\end{align*}
	\revii{Here, we have used the equation in (\ref{eq6a}) to derive the second and the third equalities.} To condense the notation, we denote 
	\[{\bf f}_{w} := (\mathcal{W}_\rho^h)^{1/2}{\bf f},\ \  {\bf c}_w := (\mathcal{W}_\rho^{2h})^{1/2}{\bf c},\]
	and use (\ref{kbar}). Finally, (\ref{eq6}) is reduced to
	\begin{multline}
	E^{n+1/2} 
	= \left(\delta {\bf f}^n_w, \left(\frac{1}{\Delta_t^2} - \frac{1}{4}\bar{K}_h + \frac{\Delta_t^2}{48}\bar{K}_h^2\right)\delta {\bf f}^n_w\right)_h
	+\left(\delta {\bf c}^n_w, \left(\frac{1}{\Delta_t^2} - \frac{1}{4}\bar{K}_{2h} + \frac{\Delta_t^2}{48}\bar{K}_{2h}^2\right)\delta {\bf c}^n_w\right)_{2h}\\
	+ \left({\bf f}^{n+1}_w + {\bf f}^n_w, \left(\frac{1}{4}\bar{K}_h- \frac{\Delta_t^2}{48}\bar{K}_{h}^2\right)\left({\bf f}^{n+1}_w + {\bf f}^n_w\right)\right)_h
	+ \left({\bf c}^{n+1}_w + {\bf c}^n_w, \left(\frac{1}{4}\bar{K}_{2h} - \frac{\Delta_t^2}{48}\bar{K}_{2h}^2\right)\left({\bf c}^{n+1}_w + {\bf c}^n_w\right)\right)_{2h},
	\end{multline}
	where we have only standard $L_2$ inner products. 
	Therefore, the fully discrete energy $E^{n+1/2} \geq 0$ if
	\[\frac{1}{\Delta_t^2} - \frac{1}{4}\bar{K}_h + \frac{\Delta_t^2}{48}\bar{K}_h^2 \geq 0,\ \ \  \frac{1}{\Delta_t^2} - \frac{1}{4}\bar{K}_{2h} + \frac{\Delta_t^2}{48}\bar{K}_{2h}^2\geq 0,\ \ \ 
	\frac{1}{4}\bar{K}_h- \frac{\Delta_t^2}{48}\bar{K}_h^2\geq 0,\ \ \ 
	 \frac{1}{4}\bar{K}_{2h} - \frac{\Delta_t^2}{48}\bar{K}_{2h}^2 \geq 0.\]
	Both $\bar{K}_h$ and $\bar{K}_{2h}$ are symmetric positive semi-definite. Therefore, all of their eigenvalues are real and nonnegative, i.e., $\kappa_j^h \geq 0$ and $\kappa_j^{2h} \geq 0$, with $\kappa_j^h$ being the j-th eigenvalue of $\bar{K}_h$ and $\kappa_j^{2h}$ being the j-th eigenvalue of $\bar{K}_{2h}$. By the fact that the eigenvalues of a matrix polynomial equals the polynomial of the eigenvalues of the matrix, 
we need
		\[\frac{1}{\Delta_t^2} - \frac{1}{4}\kappa_j^h + \frac{\Delta_t^2}{48}(\kappa_j^h)^2 \geq 0,\ \ \frac{1}{\Delta_t^2} - \frac{1}{4}\kappa_j^{2h} + \frac{\Delta_t^2}{48}(\kappa_j^{2h})^2\geq 0,\ \ 
		\frac{1}{4}\kappa^h_j- \frac{\Delta_t^2}{48}(\kappa_j^h)^2\geq 0,\ \ 
	\frac{1}{4}\kappa^{2h}_j - \frac{\Delta_t^2}{48}(\kappa^{2h}_j)^2 \geq 0\]
	for all $\kappa_j^h$ and all $\kappa_j^{2h}$ to guarantee $E^{n+1/2} \geq 0$. Following a similar calculation as in \cite{sjogreen2012fourth}, we obtain the time step restriction 
	\begin{equation}\label{dt_energy_positive}
	0 \leq \Delta_t \leq \min\left\{\frac{2\sqrt{3}}{\max_j\sqrt{\kappa_j^h}},
	\frac{2\sqrt{3}}{\max_j\sqrt{\kappa_j^{2h}}}\right\}.
	\end{equation}

\section{Terms in semi-discretization}\label{definitions}
For the first term in (\ref{elastic_semi_c}), we have 
\begin{equation*}
G_1^{2h}(N_{11}^{2h}){\bf c} := \begin{pmatrix}
\left(G_1^{2h}(N_{11}^{2h}){\bf c}\right)_1\\
\left(G_1^{2h}(N_{11}^{2h}){\bf c}\right)_2
\end{pmatrix}, \ \ \left(G_1^{2h}(N_{11}^{2h}){\bf c}\right)_p = \sum_{q =1}^2 G_1^{2h}(N_{11}^{2h}(p,q))c^{(q)},\ p = 1,2,
\end{equation*}
where $G_1^{2h}(N_{11}^{2h}(p,q))c^{(q)}$ is the second derivative SBP operator defined in (\ref{sbp_dxx}) for direction $r$. For the second term in (\ref{elastic_semi_c}), we have
\begin{equation*}
\wt{G}_2^{2h}(N_{22}^{2h}){\bf c} := \begin{pmatrix}
\left(\wt{G}_2^{2h}(N_{22}^{2h}){\bf c}\right)_1\\
\left(\wt{G}_1^{2h}(N_{22}^{2h}){\bf c}\right)_2
\end{pmatrix}, \ \ \left(\wt{G}_1^{2h}(N_{22}^{2h}){\bf c}\right)_p = \sum_{q =1}^2 \wt{G}_2^{2h}(N_{22}^{2h}(p,q))c^{(q)},\ p = 1,2,
\end{equation*}
where $\wt{G}_2^{2h}(N_{22}^{2h}(p,q))c^{(q)}$ is the second derivative SBP operator defined in (\ref{sbp_dxx_gp}) for direction $s$. For the last two terms in (\ref{elastic_semi_c}), we have 
\begin{equation*}
D_l^{2h}(N_{lm}^{2h}D_m^{2h}{\bf c}) := \begin{pmatrix}
\left(D_l^{2h}(N_{lm}^{2h}D_m^{2h}{\bf c})\right)_1\\
\left(D_l^{2h}(N_{lm}^{2h}D_m^{2h}{\bf c})\right)_2
\end{pmatrix}, \ \ \left(D_l^{2h}(N_{lm}^{2h}D_m^{2h}{\bf c})\right)_p = \sum_{q =1}^2 D_l^{2h}(N_{lm}^{2h}(p,q)D_m^{2h}c^{(q)}),\ p = 1,2,
\end{equation*}
where $D_1^{2h}c^{(q)}$ (resp. $D_2^{2h}c^{(q)}$)  is a finite difference operator in direction $r$ (resp. $s$) for the spatial first derivative defined in (\ref{sbp_dx}). For the second term in (\ref{elastic_semi_f}), we have
\begin{equation*}
{G}_2^{h}(N_{22}^{h}){\bf f} := \begin{pmatrix}
\left({G}_2^{h}(N_{22}^{h}){\bf f}\right)_1\\
\left({G}_2^{h}(N_{22}^{h}){\bf f}\right)_2
\end{pmatrix}, \ \ \left({G}_2^{h}(N_{22}^{h}){\bf f}\right)_p = \sum_{q =1}^2 {G}_2^{h}(N_{22}^{h}(p,q))f^{(q)},\ p = 1,2,
\end{equation*}
where $G_2^{h}(N_{22}^{2h}(p,q))f^{(q)}$ is the second derivative SBP operator defined in (\ref{sbp_dxx}) for direction $s$.

For the continuity of traction (\ref{continuous_traction}), we have
\[\wt{\mathcal{A}}_2^{2h}{\bf c} = N_{21}^{2h} D_1^{2h}{\bf c} + N_{22}^{2h}\wt{\mathcal{D}}_{2}^{2h}{\bf c},\]
where
\begin{equation*}
N_{21}^{2h}D_1^{2h}{\bf c} := \begin{pmatrix}
(N_{21}^{2h}D_1^{2h}{\bf c})_1\\
(N_{21}^{2h}D_1^{2h}{\bf c})_2
\end{pmatrix},\ \ (N_{21}^{2h}D_1^{2h}{\bf c})_p = \sum_{q = 1}^2 N_{21}^{2h}(p,q)D_1^{2h}c^{(q)},\ p = 1,2. 
\end{equation*}
And
\begin{equation*}
N_{22}^{2h}\wt{\mathcal{D}}_2^{2h}{\bf c} := \begin{pmatrix}
(N_{22}^{2h}\wt{\mathcal{D}}_2^{2h}{\bf c})_1\\
(N_{22}^{2h}\wt{\mathcal{D}}_2^{2h}{\bf c})_2
\end{pmatrix},\ \ (N_{22}^{2h}\wt{\mathcal{D}}_2^{2h}{\bf c})_p = \sum_{q = 1}^2 N_{22}^{2h}(p,q)\wt{\mathcal{D}}_2^{2h}c^{(q)},\ p = 1,2,
\end{equation*}
where $\wt{\mathcal{D}}_2^{2h}$ is a finite difference operator for the first spatial derivative in direction $s$ defined in (\ref{sbp_dxx_gp}); and 
\[{\mathcal{A}}_2^{h}{\bf f} = N_{21}^{h} D_1^{h}{\bf f} + N_{22}^{h}{\mathcal{D}}_{2}^{h}{\bf f},\]
where $N_{21}^hD_1^h{\bf f}$ is defined similar as $N_{21}^{2h}D_1^{2h}{\bf c}$ and $\mathcal{D}_2^h$ is a finite difference operator for the first spatial derivative in direction $s$ defined in (\ref{sbp_dxx}).

The symmetric positive semi-definite bilinear forms correspond to the fourth order scheme in the stability analysis are given as follows, 
\begin{multline*}
\mathcal{S}_{2h} ({\bf u}, {\bf v}) = \left(D_1^{2h}{\bf u}, N_{11}^{2h}D_1^{2h}{\bf v}\right)_{2hw} + \left(D_1^{2h}{\bf u}, N_{12}^{2h}D_2^{2h}{\bf v}\right)_{2hw} \\+ \left(D_2^{2h}{\bf u}, N_{21}^{2h}D_1^{2h}{\bf v}\right)_{2hw} + \left(D_2^{2h}{\bf u}, N_{22}^{2h}D_2^{2h}{\bf v}\right)_{2hw} 
+ \left({\bf u}, P^{2h}_1(N_{11}^{2h}){\bf v}\right)_{2h} + \left({\bf u}, P^{2h}_2(N_{22}^{2h}){\bf v}\right)_{2h}, 
\end{multline*}
where $P_1^{2h}(N_{11}^{2h})$ (resp. $P_2^{2h}(N_{22}^{2h})$) is a positive semi-definite operator for direction $r$ (resp. $s$) and is small for smooth grid functions but non-zero for odd-even modes, see \cite{sjogreen2012fourth} for details.
	\begin{multline*}
\mathcal{S}_{h} ({\bf u}, {\bf v}) = \left(D_1^{h}{\bf u}, N_{11}^{h}D_1^{h}{\bf v}\right)_{hw} + \left(D_1^{h}{\bf u}, N_{12}^{h}D_2^{h}{\bf v}\right)_{hw} \\+ \left(D_2^{h}{\bf u}, N_{21}^{h}D_1^{h}{\bf v}\right)_{hw} + \left(D_2^{h}{\bf u}, N_{22}^{2h}D_2^{h}{\bf v}\right)_{hw} 
+ \left({\bf u}, P^{h}_1(N_{11}^{h}{\bf v})\right)_{h} + \left({\bf u}, P^{h}_2(N_{22}^{h}{\bf v})\right)_{h}.
\end{multline*}
Here, $P_i^{h}(N_{ii}^{h})$ is defined similar as $P_i^{2h}(N_{ii}^{2h})$ in $\mathcal{S}_{2h}({\bf u}, {\bf v})$.

The symmetric positive semi-definite bilinear forms correspond to the sixth order scheme in the stability analysis are given as follows, 
\[\mathcal{S}_{2h} ({\bf u}, {\bf v}) = \left(D_1^{2h}{\bf u}, N_{11}^{2h}D_1^{2h}{\bf v}\right)_{2hw} \!\!+ \left(D_1^{2h}{\bf u}, N_{12}^{2h}D_2^{2h}{\bf v}\right)_{2hw} \!\!+ \left(D_2^{2h}{\bf u}, N_{21}^{2h}D_1^{2h}{\bf v}\right)_{2hw} \!\!+ \left(D_2^{2h}{\bf u}, N_{22}^{2h}D_2^{2h}{\bf v}\right)_{2hw},\]
and
\[\mathcal{S}_{h} ({\bf u}, {\bf v}) = \left(D_1^{h}{\bf u}, N_{11}^{h}D_1^{h}{\bf v}\right)_{hw} + \left(D_1^{h}{\bf u}, N_{12}^{h}D_2^{h}{\bf v}\right)_{hw} + \left(D_2^{h}{\bf u}, N_{21}^{h}D_1^{h}{\bf v}\right)_{hw} + \left(D_2^{h}{\bf u}, N_{22}^{2h}D_2^{h}{\bf v}\right)_{hw}.\]
We note that the different definitions of $\mathcal{S}_{2h}({\bf u}, {\bf v})$, $\mathcal{S}_h({\bf u}, {\bf v})$ in the fourth order and the sixth order method is because we construct fourth order SBP operators by the way proposed by Sj\"{o}green and Petersson \cite{sjogreen2012fourth} and sixth order SBP operators by the way proposed by Mattsson \cite{Mattsson2012}. But in both cases, $\mathcal{S}_{2h}({\bf u}, {\bf v})$, $\mathcal{S}_h({\bf u}, {\bf v})$ are positive semi-definite.

%
%
%
%
%
%
%
%
%
%
%
%
%
%
%
%

\bibliography{lu}
\bibliographystyle{plain}

\end{document}